\documentclass[10pt]{article}
\usepackage{amsmath, amsthm}
\usepackage{amscd}
\usepackage{amsfonts}
\usepackage{amssymb}
\usepackage{mathrsfs}
\usepackage{pictex}

\makeindex

\newtheorem{theorem}{Theorem}[section]

\newtheorem{lemma}[theorem]{Lemma}
\newtheorem{corollary}[theorem]{Corollary}

\theoremstyle{definition}

\def\proof{{\noindent\sc Proof. \quad}}
\newcommand\proofof[1]{{\noindent\sc Proof of #1. \quad}}
\def\eproof{{\mbox{}\hfill\qed}\medskip}

\makeatletter

\renewcommand{\bar}{\overline}

\renewcommand{\hat}{\widehat}
\renewcommand{\bar}{\overline}
\renewcommand{\tilde}{\widetilde}

\def\f{{\cal F}} 

\def\1{\mbox{I\hspace{-.6em}1}} 
\def\Var{\mathop{\sf Var}}
\def\dist{\mathsf{dist}}

\def\N{\mathbb{N}}
\def\Z{\mathbb{Z}}
\def\R{\mathbb{R}}

\def\C{\mathbb{C}}

\def\E{\mathbb{E}}

\def\ll{{[\kern-1.6pt [}}
\def\rr{{]\kern-1.4pt ]}}
\def\bll{{\biggl[\kern-3pt \biggl[}}
\def\brr{{\biggr]\kern-3pt \biggr]}}

\def\dist{{\rm dist}}

\def\Id{{\rm Id}}

\def\cruz{\raise0.7pt\hbox{$\scriptstyle\times$}}

\def\alto{\rule[-2mm]{0mm}{6.5mm}}
\def\bajo{\rule[-2.5mm]{0mm}{2mm}}

\def\JACM{Journal of the ACM}

\def\TCS{Theoret. Comp. Sci.}

\def\JAMS{Journal of the Amer. Math. Soc.}

\def\JoC{J. Compl.}
\def\MP{Math. Program.}
\def\scC{{\mathscr C}}

\def\bD{{\mathbf D}}

\def\Hd{\HH_{\mathbf d}}

\def\P{\mathbb P}

\newcommand{\HH}{\ensuremath{\mathcal H}}

\begin{document}

\begin{title}
{\LARGE {\bf A Numerical Algorithm for Zero Counting. \\
III: Randomization and Condition}}
\end{title}
\author{Felipe Cucker
\thanks{Partially supported by GRF grant City University
100810.}\\
Dept. of Mathematics\\
City University of Hong Kong\\
HONG KONG\\
e-mail: {\tt macucker@cityu.edu.hk}
\and
Teresa Krick
\thanks{Partially supported by grants ANPCyT 33671/05,
UBACyT X113/2008-2010 and CONICET PIP/2010-2012.}\\
Departamento de Matem\'atica\\
Univ. de Buenos Aires \&  IMAS, CONICET\\
ARGENTINA\\
e-mail: {\tt krick@dm.uba.ar}
\and
Gregorio Malajovich\thanks{Partially supported by
CNPq grants 470031/2007-7,
303565/2007-1, and by FAPERJ.
}\\
Depto. de Matem\'atica Aplicada\\
Univ. Federal do Rio de Janeiro\\
BRASIL\\
e-mail: {\tt gregorio@ufrj.br}
\and
Mario Wschebor\\
Centro de Matem\'{a}tica\\
Universidad de la Rep\'{u}blica\\
URUGUAY\\
e-mail: {\tt wschebor@cmat.edu.uy}
}

\date{}
\makeatletter
\maketitle
\makeatother

\begin{quote}
{\small
{\bf Abstract.}
In a  recent paper~\cite{CKMW1} we analyzed a
numerical algorithm for computing the number of real zeros
of a polynomial system. The analysis relied on a condition
number $\kappa(f)$ for the input system $f$. In this paper
we look at $\kappa(f)$ as a
random variable derived from imposing a probability measure
on the space of polynomial systems and give bounds for both the tail
$\P\{\kappa(f)> a\}$ and the expected value $\E(\log\kappa(f))$.
}\end{quote}

\begin{quote}
\end{quote}

\begin{quote}
\end{quote}

\section{Introduction}\label{intro}

\subsection{Overview}

This paper is the third  of a series which started
with~\cite{CKMW1,CKMW09}. In the first paper of the series
we analyzed  a numerical algorithm for computing the number
of real zeros of a polynomial system. This algorithm works with
finite precision and the analysis provided bounds for both its
complexity (total number of arithmetic operations) and the
machine precision needed to guarantee that the returned
value is correct. Both bounds depended on size parameters
for the input system $f$ (number of polynomials, degrees, etc.)
as well as on a condition number $\kappa(f)$ for $f$. A
precise statement of the main result in~\cite{CKMW1}
is Theorem~1.1 therein. To the best of our knowledge, this
theorem is the only result providing a finite-precision analysis
of a zero counting algorithm. Consequently, as of today, to
understand zero-counting computations in the presence of
finite-precision appears to require an understanding of $\kappa(f)$.

Unlike the aforementioned size parameters, the condition
number $\kappa(f)$ cannot be read directly from the system $f$.
Indeed, it is conjectured that the computation of $\kappa(f)$
is at least as difficult as solving the zero counting problem for $f$,
so we need a much depper understanding of $\kappa(f)$.
In the second paper of the series~\cite{CKMW09}, we attempted
to provide such an understanding from two different angles.
Firstly, we showed that a closely related condition number
$\tilde{\kappa}(f)$ satisfies a Condition Number Theorem, i.e.,
$\tilde{\kappa}(f)$ is the normalized inverse of the distance
from $f$ to the set of ill-posed systems (those having multiple
zeros). The relation between
the quantities $\kappa(f)$ and $\tilde\kappa(f)$ is close indeed
(see~\cite[Prop.~3.3]{CKMW09}):
\begin{equation*}
  \frac{\tilde\kappa(f)}{\sqrt n} \le \kappa(f)\le
   \sqrt{2n}\ \tilde\kappa(f).
\end{equation*}
Secondly, we used this characterization, in conjunction with a
result from~\cite{BuCuLo:07}, to provide a smoothed analysis
of $\tilde{\kappa}(f)$ (and hence, of $\kappa(f)$ as well). A smoothed
analysis of the complexity and accuracy for the algorithm
in~\cite{CKMW1} immediately follows. Details about smoothed
analyses and distance to ill-posedness can be found in the
introduction of~\cite{CKMW09}.

As a consequence of the smoothed analysis of $\tilde{\kappa}(f)$
one immediately obtains an average-case analysis of this condition
number. One is left, however, with the feeling that the bounds
thus obtained are far from optimal. Indeed, these bounds follow
from a result which is general in two aspects. Firstly, it is
a smoothed analysis (of which usual average analysis is just a
particular case). Secondly, it is derived from a very general result
yielding smoothed analysis bounds for condition numbers satisfying
a Condition Number Theorem and stated in terms of some geometric
invariants (degree and dimension) of the set of ill-posed inputs.
The question of whether a finer average analysis can be obtained by
using methods more ad-hoc for the problem at hand naturally
poses itself.

In this paper we show that such bounds are possible. Loosely
speaking, the average analysis in~\cite{CKMW09} shows a bound
for a typical $\tilde{\kappa}(f)$ - or $\kappa (f)$ - which is of order  ${\mathcal D}^2$
where ${\mathcal D}$ is the B\'ezout number of $f$. Here we show
that $\sqrt{\mathcal D}$ is
a more accurate upper-bound. This improvement is meaningful, since $\mathcal D$ increases exponentially with $n$.
Our main result implies that if the maximum degree $\bD$ remains bounded as $n$
grows, $\E (\ln \kappa (f)) $ is bounded from above by a quantity
equivalent to  $\ln (\mathcal{D}^{1/2})$, which according to the
Shub-Smale Theorem, see \cite{Bez2}, equals the logarithm of the mathematical expectation
of the total number of real roots of the polynomial system. More
precisely,
$$
\limsup_{n\to \infty}  \frac{\E (\ln \kappa (f))}{\ln (
\mathcal{D}^{1/2})}\le 1.
$$

No non-trivial lower bound has been obtained for the time being as far as we know.

We next proceed to set up the
basic notions and notations enabling us to state the above in
more precise terms.

\subsection{Basic definitions and main result}\label{sec:notations}

For $d\in \N$ we denote by $\HH_d$ the subspace of
$\R[x_0,\ldots,x_n]$ of homogeneous polynomials of degree $d$ and,
for $\mathbf{d}:=(d_1,\dots,d_n)$, we set $\Hd:=\HH_{d_1}\times
\cdots\times \HH_{d_n}$. We endow $\HH_d$ with the Weyl norm
which is defined, for $f\in\HH_d$, $f(x)=\sum_{|j|=d}a_jx^j$, by
$$
   \|f\|_W^2=\sum _{|j|=d}\frac{a_j^2}{{d\choose j}}
$$
where $x=(x_0,\dots,x_n)$, $j=(j_0,\dots,j_n)$,
$|j|:=j_0+\cdots+j_n$, $x^j=x_0^{j_0}\cdots x_n^{j_n}$ and
$ {d\choose j}:=\frac{d!}{j_0!\cdots j_n!}$. We then endow
$\Hd$ with the norm given by
$$
    \|f\|:=\displaystyle{\max_{1\le i\le n}\|f_i\|_W}.
$$

For $f=(f_1,\dots,f_n)\in \Hd$, as in \cite{CKMW1}, we define the
following condition number
$$
   \kappa(f)=\max_{x\in S^n} \min
     \left\{\mu_{\rm norm}(f,x),\frac{\|f\|}{\|f(x)\|_\infty}\right\}
$$
with
$$
   \mu_{\rm norm}(f,x)
  =\sqrt{n}\,\|f\|\, \left\| D_x(f)^{-1} M \right\|.
$$
Here
\begin{itemize}
\item
$D_x(f)=Df(x)|_{T_xS^n}$ is
the derivative of $f$ along the unit sphere $S^n\subset\R^{n+1}$
at the point $x$, a linear operator
from the tangent space $T_x(S^n)$ to $\R^n$,
\item
$M:={\scriptstyle \left[
  \begin{array}{ccc} \sqrt{d_1} \\
   &  \ddots & \\ &  & \sqrt{d_n}
  \end{array} \right]}$
is the scaling $n\times n$ diagonal matrix with diagonal entries
the square roots of the degrees $d_i=\deg(f_i)$,
\item
the norm $\|D_x(f)^{-1} M\|$ is the spectral
norm,  i.e., the operator norm
$\max\{ \|D_x(f)^{-1} M\, y\|_2; y\in S^n, y\perp x\}$ with
respect to $\|\ \|_2$,
\item
$\|f(x)\|_\infty =\max_{1\le i\le n} |f_i(x)|$ denotes as usual the
infinity norm.
\end{itemize}
\bigskip

We next impose the probability measure on $\Hd$ defined
by Eric Kostlan~\cite{KostlanPhD} and Shub-Smale~\cite{Bez2}.
This measure assumes the
coefficients of the polynomials $f_i=\sum_{|j|=d_i} a_j^{(i)} x^j$
are independent, Gaussian, centered random variables, with variances
$$
      \Var(a_j^{(i)})={  d_i \choose j }.
$$
For this distribution, and for $x,y \in \R ^{n+1}, 1\le i,k\le n$,
covariances are given by (see Lemma \ref{lemcovar} below)
$$
  \E \big( f_i(x)f_{k}(y)\big)= \delta_{ik}\langle x,y \rangle ^{d_i}
$$
where $\delta_{ik}$ is the Kronecker symbol.

This probability law is invariant under the action of the
orthogonal group and permits to perform the computations below,
which appear to be much more complicated under other distributions
not sharing this invariance property.
\medskip

To state our main results a number of quantities will be
useful. We use the notation
$$
  \bD:=\max_{1\le i\le n} d_i,\quad \mathcal{D}=\prod_{i=1}^n d_i, \quad N := \dim \Hd = \sum_{i=1}^n{n+d_i\choose n}.
$$
We note that   $\mathcal{D}$ is the B\'{e}zout number of the
polynomial system. We may assume here that $d_i\ge 2$ for $1\le i\le
n$ since otherwise we could restrict to a system with fewer equations
and unknowns. Notice that
  $N\le n^{\bD+2}$.

\bigskip

We are now ready to state our main result.

\begin{theorem}\label{boundkapa}
Let the random system $f$ satisfy the conditions of the Shub-Smale
model and assume $n\geq 3$. Then,
\begin{description}
\item[(i)]
For $a > 4\sqrt 2\,\bD^2n^{7/2}N^{1/2}$ one has
$$
  \P \big( \kappa (f)>a \big)\leq K_n \frac{\sqrt{2n}(1+\ln
  (a/\sqrt{2n}))^{1/2}}{a},
$$
where $K_n:=8\bD^2{\mathcal{D}}^{1/2}\,{N}^{1/2}n^{5/2}+1$.
\item[(ii)]
$$
  \E (\ln \kappa (f))\leq \ln K_n + (\ln K_n)^{1/2}+(\ln
  K_n)^{-1/2} + \frac{1}{2}\ln(2n).
$$
\end{description}
\end{theorem}
\bigskip

In fact we are going to prove the corresponding result for the
alternative quantity $\tilde\kappa(f)$ already considered in
\cite{CKMW09}, since it will enable us  to use $\mathbb{L}^2$
methods, which are more adapted to the type of calculations we will
perform. We recall that
$$
   \tilde\kappa(f)=\frac{\|f\|_W}
   {\big(\min_{x\in S^n} \{\| D_x(f)^{-1}M\|^{-2}
          +\|f(x)\|_2^2\}\big)^{1/2} }
$$
where $\|f\|^2_W:=\sum_{1 \le i\le n} \|f_i\|_W^2$ is the Weyl norm
of the system and $\|f(x)\|^2:=\sum_{1\le i\le n} f_i(x)^2$ denotes
the usual Euclidean norm. As we have already mentioned,
we have
$\frac{\tilde\kappa(f)}{\sqrt n} \le \kappa(f)\le\sqrt{2n}\ \tilde\kappa(f)$.
Also,  as a consequence of \cite[Th.~1.1]{CKMW09}, $\tilde \kappa(f)$
 satisfies $\tilde \kappa(f)\ge 1$ for all $f \in \Hd$.\\

We will therefore obtain Theorem~\ref{boundkapa} as a direct
consequence of the following result.

\begin{theorem}\label{boundkapatilde}
Let the random system $f$ satisfy the conditions of the Shub-Smale
model and assume $n\geq 3$. Then,

\begin{description}
\item[(i)]
For $a>4\,\bD^2n^3N^{1/2}$ one has
$$
\P \big( \tilde\kappa (f)>a \big)\leq K_n \frac{(1+\ln a)^{1/2}}{a}
$$
where $K_n:=8\bD^2{\mathcal{D}}^{1/2}\,{N}^{1/2}n^{5/2}+1$.

\item[(ii)]
$$
     \E (\ln \tilde\kappa (f))\leq \ln K_n + (\ln K_n)^{1/2}+(\ln K_n)^{-1/2}.
$$
\end{description}
\end{theorem}

\medskip \noindent  Theorem~\ref{boundkapa} follows from  $\P \big(\kappa(f)> a\big)\le
\P\big(\tilde \kappa(f)> a/\sqrt{2n}\big)$, since $\kappa(f)>a
\Rightarrow \tilde\kappa(f)  >a/\sqrt{2n}$.\\

\medskip

The proof of Theorem~\ref{boundkapatilde} is given in
Section~\ref{demostracion}. It requires a
certain number of auxiliary results. With the aim of isolating (and in this
way highlighting) the main ideas, we will postpone the proof of these auxiliary results  to
Section~\ref{auxlemmas}, though  stating them  as needed in the text.  This  will be indicated  by the symbol
$\diamondsuit$ at the end of the statement.

\subsection{Relations with previous work}

Probably the most successful combination of algorithmics, conditioning,
and probability occurs in the study of complex polynomial systems (a
setting similar to ours but with the coefficients of the polynomials now
drawn from $\C$ and considering projective complex zeros). This  study
spans an impressive collection of papers, which began
with~\cite{Bez1,Bez2,Bez3,Bez4,Bez5} and continued
in~\cite{BePa08} and~\cite{Bez6,Bez7}.
The final outcome of these
efforts is a randomized algorithm producing an approximate zero of
the input system in  expected time which is polynomial in the size
of the system. The expectation is with respect to {\em both} the
random choices in the algorithm and a probability measure on the
input data.

The condition number of a system $f$ in this setting is defined to be
\begin{equation*}\label{muC}
  \mu_{\rm norm}(f):=\max_{\zeta\in S^n_{\C}\mid f(\zeta)=0}
  \mu_{\rm norm}(f,\zeta).
\end{equation*}
Here $\mu_{\rm norm}(f,\zeta)$ is roughly the quantity we defined
above.  Over the reals, it may not be well-defined since the zero
set of $f$ may be empty. If one restricts attention to the subset
${\mathcal R}_{\mathbf{d}}\subset\Hd$ of those systems having at
least a real zero one may similarly define a measure $\mu_{\rm
worst}(f)$, maximizing over the set of real zeros. This has been
done in~\cite{BoPa08} where bounds for the tail and the expected
value of $\mu_{\rm worst}(f)$ are given. These bounds are very
satisfying (for instance, the tail $\P \big(\mu_{\rm worst}>a \big)$
is bounded by an expression in $a^{-2}$, a fact ensuring the
finiteness of $\E(\mu_{\rm worst}(f))$). The measure $\mu_{\rm
worst}(f)$, however, is hardly a condition number for the problem of
real zeros counting, not even restricted to the subset ${\mathcal
R}_{\mathbf{d}}$. To understand why, consider a polynomial as in the
left-hand side of the figure below.
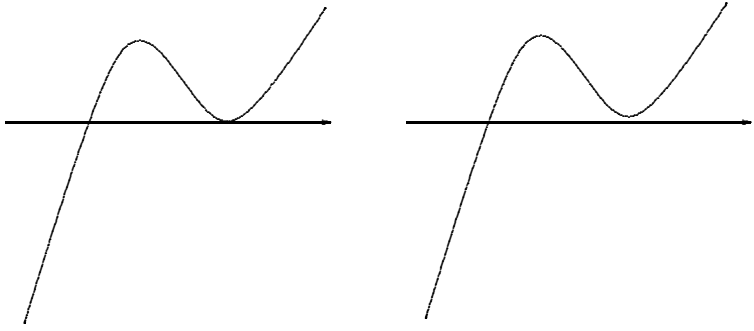
\begin{center}
  \input curva_cond.pictex
\end{center}
For this polynomial one has $\mu_{\rm worst}=\infty$.

An upward small perturbation (as in the right-hand side) yields a
low value of $\mu_{\rm worst}$. This value admits a finite limit
when such perturbations are small enough!
The measure
$\mu_{\rm worst}(f)$ appears to be insensitive to the closeness
to ill-posedness. This runs contrary to the notion of
conditioning~\cite{Demmel87,FreundVera,Rump,Wilkinson72}.

A condition number $\mu^*(f)$ for the feasibility problem of real
systems (which, obviously, needs to be  defined on all of $\Hd$)
was given in~\cite{CS98} by taking
$$
   \mu^*(f)=\left\{\begin{array}{ll}
       \displaystyle\min_{\zeta\in S^n\mid f(\zeta)=0}
       \mu_{\rm norm}(f,\zeta) &
       \mbox{if $f\in {\mathcal R}_{\mathbf{d}}$}\\  [10pt]
      \displaystyle\max_{x\in S^n} \frac{\|f\|}{\|f(x)\|}  &
       \mbox{otherwise.}
       \end{array}\right.
$$
As of today, there is no probabilistic analysis for it.

\section{Proof of Theorem \ref{boundkapatilde}}\label{demostracion}

The proof relies on the so-called Rice Formula for the expectation
of the number of local minima of a real-valued random field. This is
described precisely in Step 2 below. Previously, in Step 1, we use
large deviations to show that for large $n$, except on a set of small probability, the numerator $\|f\|_W$ in  $\tilde\kappa(f)$ is nearly equal to $N^{1/2}$. Steps 3, 4, and 5 estimate the different expressions
occurring in Rice formula. Finally, Step 6 wraps up all these
estimates to yield the upper bound for the density and Step~7
derives from it the  bounds claimed in the statement of
Theorem~\ref{boundkapatilde}.\\

\noindent During the rest of the proof, we set
$\underline{L}=\underline{L}(f):=\min_{x\in S^n}
\{\|D_x(f)^{-1}M\|^{-2}+\|f(x)\|_2^2 \}$ so that
$\tilde\kappa(f)=\|f\|_W/\sqrt{\underline{L}}$. We observe that
$$
    \| D_x(f)^{-1}M\|^{-1}
   =\sigma_{\min} (M^{-1}D_x(f))
   =\min\{\|M^{-1}D_x(f)y\|: y\in S^n, y\perp x\},
$$
(where $\sigma _{\min}$ denotes the minimum singular value),
and therefore
$$
   \underline{L}=
  \min\{\|M^{-1}D_x(f)y\|^{2}+\|f(x)\|_2^2:
   x,y\in S^n, y\perp x \}
$$
is the minimum of the random field $\{L(x,y): (x,y)\in V\}$ where
\begin{eqnarray} \label{randomfield}
   L(x,y)&:=&\|M^{-1}D_x(f)y\|^{2}+\|f(x)\|_2^2,
  \nonumber  \\
  &=&\sum_{i=1}^n \frac{1}{d_i}\left(\sum_{j,k=0}^n
       \partial_jf_i(x)\partial_kf_i(x)y_jy_k\right) +
       \sum_{i=1}^n f_i^2(x); \\
  \mbox{and}\qquad
V  &:=&  \{(x,y)\in \R^{n+1}\times\R^{n+1}: \|x\|=\|y\|=1,
   \langle x,y\rangle=0\}. \nonumber
 \end{eqnarray}
Here $y=(y_0,\dots,y_n)$ and, for $1\le i\le n$ and  $0\le j\le n$,
$\partial_jf_i(x)$ denotes the partial derivative of $f_i$
with respect to $x_j$ at the point $x$.

\bigskip
\noindent\textbf{Step 1.} Our first step consists in replacing the
Weyl norm in the numerator of $\tilde\kappa(f)$ by a non-random
constant, at the cost of adding a
small probability, which will be controlled using large deviations.\\

\noindent Let $a>1$. We have
$$
     \P \left(\tilde\kappa (f)>a \right)
 = \P \left(\frac{\underline{L}}{\|f\|_W^2}<\frac{1}{a^2} \right)
 \leq \P\left(\underline{L}<\frac{1}{a^2} (1+\ln a )N \right)
     +\P\Big(\|f\|_W^2 \geq (1+\ln a)N \Big).
$$
We bound the second term in the right-hand side above using the following result that will be proved in Section~\ref{auxlemmas}.
\begin{lemma}\label{largedev}
Set
$$
   N:=\dim \Hd= \sum_{i=1}^n{
                    n+d_i \choose
                    n }
$$
Then, for $\eta >0$,
$$
   \P \left(\|f\|_W^2 \geq (1+\eta)N \rule{0em}{3ex}\right)
  \leq e^{-\frac{N}{2}(\eta-\ln(\eta +1))}.\qquad  \diamondsuit
$$
\end{lemma}

\noindent Therefore, setting $\eta=\ln a$, we  obtain
\begin{equation}\label{des1}
   \P \left(\tilde\kappa (f)>a \right)\leq \P
   \left(\underline{L}<\frac{1}{a^2} (1+\ln a )N \right)+
   \exp\left(-\frac{N}{2}(\ln a - \ln (\ln a +1)\right).
\end{equation}
The second term in the right-hand side above can be easily
estimated. We therefore turn our attention to the first. Given
$\alpha>0$, we want to compute an upper bound for
$$
     \P \left(\underline{L}<\alpha \right).
$$

\bigskip \noindent\textbf{Step 2}. Our second step consists in giving a bound
 for the density function $p_{\underline{L}}(u)$  of the
random variable $\underline{L}$, i.e.  such that $$ \P
\left(\underline{L}<\alpha \right)=\int_{0}^{\alpha}
p_{\underline{L}}(u)du$$ since $\underline{L}$ is non-negative. We recall that the quantity $\underline{L}$ is the minimum
of the random field $\{ L(x,y):(x,y) \in V \}$,  for $L$ and $V$
defined in Formula~(\ref{randomfield}).\\

\noindent Notice that $V$ is the Stiefel manifold $S(2,n+1)$, a
compact, orientable, $\scC^{\infty}$-differentiable manifold of
dimension $2n-1$, embedded in $\R^{n+1}\times \R^{n+1}$. For each
linear orthogonal transformation $U$ of $\R^{n+1}$, define
$\tilde{U}:V\rightarrow V$, $(x,y)\mapsto (Ux,Uy)$, and denote by
$\tilde{\mathcal{U}}$ the set of these $\tilde{U}$ provided with the
group structure naturally inherited from the orthogonal group in
$\R^{n+1}$. Then $\tilde{\mathcal{U}}$ acts transitively
on $V$.\\

\noindent At a generic point $(x,y)$ of the manifold $V$, the normal
space $N_{(x,y)}(V)$ has dimension $(2n+2)-(2n-1)=3$, and  is
generated by the orthonormal set
$\left\{(x,0),(0,y),\frac{1}{\sqrt{2}}(y,x) \right\}$. Therefore, if
$\{z_2,\ldots,z_n\} \subset \R^{n+1}$ is such that
$\{x,y,z_2,\ldots,z_n \}$ is an orthonormal basis of $\R^{n+1}$, the
set
\begin{equation}\label{baseestan}
  \mathcal{B}_{T_{(x,y)}}:=
   \left\{(z_2,0),\ldots,(z_n,0),(0,z_2),\ldots,(0,z_n),
   \frac{1}{\sqrt{2}}(y,-x)\right\}
\end{equation}
is an orthonormal basis of the tangent space $T_{(x,y)}(V)$.\\

\noindent We denote by $\sigma _V\big(d(x,y)\big)$ the geometric
measure on $V$ (i.e. the measure induced by the Riemannian distance
on $V$), which is invariant under the action of the group
$\tilde{\mathcal{U}}$.  The total measure satisfies
\begin{equation}\label{volumenV}
    \sigma _V (V)=\sqrt{2}\sigma_{n-1}\sigma _n,
\end{equation}
where $\sigma _k=2\pi^{(k+1)/2}/\Gamma ((k+1)/2) $ is the
total $k$-th dimensional measure of the unit sphere $S^k$, see for
example~\cite[Lemma~13.5]{AW2}.\\

\noindent For $\alpha >0$ and $S$ a Borel subset of $V$, we denote by $m_{\alpha}(L,S)$ the number of
local minima of the random function $L $ on the set $S$, having
value smaller than $\alpha$. Clearly:

\begin{equation}\label{distmin}
\P(\underline{L}<\alpha)=\P\big(m_{\alpha}(L,V) \geq 1\big)\leq
\E\big(m_{\alpha}(L,V)\big).
\end{equation}
Our aim is to give a useful expression for the right-hand side of
Formula~(\ref{distmin}). For that purpose, let us set for each Borel
subset $S$ of $V$, $\nu (S):=\E\big(m_{\alpha}(L,S)\big)$. Clearly,
$\nu$ is a measure. The invariance of the law of the random field
$\{L(x,y):(x,y)\in V\}$ under the action of
$\tilde{\mathcal{U}}$ implies that $\nu $ is also invariant under $\tilde{\mathcal{U}}$.\\

\noindent Let $\psi:B_{2n-1,\delta}\rightarrow \R^{n+1}\times \R^{n+1} $ be a chart on $V$, that is, a smooth
diffeomorphism between the ball in $\R^{2n-1}$ centered at the origin with radius $\delta >0$
and its image $W=\psi (B_{2n-1,\delta})\subset V$.\\

\noindent We denote by $\tilde{L}:B_{2n-1,\delta}\rightarrow \R$ the
composition $\tilde{L} (w)=L\big(\psi (w)\big)$.\\

\noindent As we already mentioned, our main tool is Rice formula, of which we now present a quick overview:\\

\noindent Let $U$ be an open subset of $\R^n$ and $Z:U\rightarrow \R^n$ a random function having sufficiently smooth paths. Let us denote by $\nu ^Z(S)$ the number of zeros of $Z$ belonging to the Borel subset $S$ of $U$. Under certain general conditions on the probability law of $Z$, one can compute the expectation of $\nu ^Z(S)$ by means of an integral on the set $S$. The integrand is a certain function depending on the underlying probability law.\\

\noindent The simplest form of such a formula is the following:

\begin{equation}\label{rice0}
\E (\nu ^Z(S))=\int_S~\E \big(|\det (Z'(t))|\big/Z(t)=0 \big)\,p_{Z(t)}(0) \,dt
\end{equation}
One must be careful in the choice of the version of the conditional expectation and the density $p_{Z(t)}(\cdot)$ of the random vector $Z(t)$, since they are only defined almost everywhere. But this can be done in a certain number of cases in a canonical form, in such a way that the formula holds true.\\

\noindent This kind of formula can be extended to a variety of situations, such as: a) the zeros of $Z$ can be ``marked'', which means that instead of all zeros, we count only those zeros satisfying certain additional conditions; b) the domain can be a manifold instead of an open subset of Euclidean space; c) one has formulas similar to (\ref{rice0}) for the higher moments of $\nu ^Z(S)$; d) the dimension of the domain can be larger than the one of the image, in which case the natural problem, instead of counting roots, is studying the geometry of the random set $Z^{-1}(\{ 0 \})$. For a detailed account of this subject, including proofs and applications, see \cite[Chapters 3 and 6]{AW2}.\\

\noindent Here we want to express by means of a Rice formula  the expectation
$$
\nu(S)= \E \big(m_{\alpha}(L,S)\big)=\E \big(m_{\alpha}(\tilde{L},\psi ^{-1}(S))\big)
$$
In our case, with probability $1$,  $m_{\alpha}(\tilde{L},\psi ^{-1}(S))$ equals the number of points $w \in \psi ^{-1}(S)$ such that the derivative $\tilde{L}'(w)$ vanishes,  the second derivative $\tilde{L}''(w)$ is positive definite and the value $L(w)$ is bounded by $\alpha.$ Then, under certain conditions, we can write (use   \cite[Formula~(6.19)]{AW2}, mutatis mutandis):

\begin{equation}  \label{riceminalfa}
\aligned
\nu&(S)= \E \big(m_{\alpha}(L,S)\big)=\E \big(m_{\alpha}(\tilde{L},\psi ^{-1}(S))\big)\\
&=\int_0^{\alpha}du  \int_{\psi ^{-1}(S)}\E
\left(\big|\det(\tilde{L}''(w))\big| \chi_{\{\tilde{L}''(w)\succ 0
\}} /\tilde{L}(w)=u, \tilde{L}'(w)=0 \right)
  p_{\tilde{L}(w),\tilde{L}'(w)}(u,0)~dw.
\endaligned
\end{equation}
Here $\chi_{A}$ means indicator function of the set $A$, $\succ$
means positive definite, $p_{\tilde{L}(w),\tilde{L}'(w)}$ is the joint
density in $\R^1  \times \R^{2n-1}$ of the pair of random variables
$\big( \tilde{L}(w),\tilde{L}'(w)\big)$, and $dw$ is Lebesgue measure
on $ \R^{2n-1}$. Note that in the chart image, $d\sigma _V=\big(\det
\big((\psi ' (w))^t\psi '(w)\big)\big)^{1/2}dw.$ \\

\noindent In \cite[Proposition~6.6]{AW2} it is proved that if the
integrand in Formula~(\ref{riceminalfa}) were well-defined then the
 change of variable formula would be satisfied, so that $\nu(S)$
would be the integral of a $(2n-1)$-form. In that case,
Formula~(\ref{riceminalfa}) would already imply that the measure
$\nu $ is finite and absolutely continuous with respect to $\sigma
_V$, so that one could write for each Borel subset $S$ of $V$

$$
\nu (S)=\int _S~g~d\sigma _V
$$
for a continuous function  $g$. Let us prove that in that case the
Radon-Nikodym derivative $g$ would be constant. To see this, notice
that $\sigma _V$ is also invariant under $\tilde{\mathcal{U}}$ and
the action of this group is transitive on $V$. If $g$ takes
different values at two points $(x_1,y_1)$ and $ (x_2,y_2)$ of $V$,
letting $\tilde{U} \in \tilde{\mathcal{U}}$ be such that
$\tilde{U}(x_1,y_1)=(x_2,y_2)$, we  can find a small neighborhood
$S$ of $(x_1,y_1)$ such that
$$
\int _S~g~d\sigma _V~\neq ~ \int _{\tilde{U}(S)}~g~d\sigma _V,
$$
 contradicting the invariance of $\nu$.\\

\noindent We could then compute the constant $g$ by computing it at
the point $(e_0,e_1)$. We choose the chart $\psi$ in such a way that
$\psi (0)=(e_0,e_1)$ and $\big(\psi ' (0)\big)^t\psi '(0)=I_{2n-1}$
and compute

\begin{equation*}
\aligned g&=\lim_{\varepsilon\rightarrow 0}\frac{\nu(\psi
(B_{2n-1,\varepsilon}))}{\sigma _V(\psi
(B_{2n-1,\varepsilon}))}
\\
&=\int_0^{\alpha}  \E \left(\big|\det(\tilde{L}''(0))\big|
\chi_{\{\tilde{L}''(0)\succ 0 \}} /|\tilde{L}(0)=u, \tilde{L}'(0)=0
\right) p_{\tilde{L}(0),\tilde{L}'(0)}(u,0)~du
\endaligned
\end{equation*}

\noindent So, if Formula~(\ref{riceminalfa}) were true, it follows
that we could write
\begin{equation}\label{riceinicial}
\nu(S)= \sigma_V(S)\int_0^{\alpha}  \E
\left(\big|\det(\tilde{L}''(0))\big| \chi_{\{\tilde{L}''(0)\succ 0
\}} /\tilde{L}(0)=u, \tilde{L}'(0)=0 \right)
p_{\tilde{L}(0),\tilde{L}'(0)}(u,0)~du.
\end{equation}
However, if one computes the ingredients in the integrand of the
right-hand side of Formula~(\ref{riceminalfa}), it turns out that
the value of the density is $+\infty$ and the conditional
expectation
vanishes. So, the formula is meaningless in this form.\\

\noindent
To overcome this difficulty we proceed as follows:\\

\noindent Let $S_{(x,y)}=\mbox{span}(z_2,\dots,z_n)\subset \R^{n+1}$
be the orthogonal complement of $\mbox{span}(x,y) \subset  \R
^{n+1}$ and $\pi _{x,y}:\R^{n+1} \to S_{(x,y)}$ be the orthogonal
projection. For $(x,y)\in V$, we introduce a new random vector
$\zeta_{(x,y)}$ defined as
\begin{equation}\label{zetaxv}
   \zeta_{(x,y)}:=\Big( \big( \pi _{x,y} ( f'_i(x)),\partial_{yy}f_i(x) \big) ,
   1\le i\le n \Big) \ \in \ \big(S_{(x,y)}\times \R\big)^n \ \cong \
   \R^{n^2},
\end{equation}
where for $1\le i\le n$, $f'_i(x)$ is the free derivative (the
gradient) of $f_i$ at $x$, the first $(n-1)$ coordinates are given
by the coordinates of the projection of $f'_i(x)$ onto $S_{(x,y)}$
in the orthonormal basis $ \{z_2,\dots,z_n\}$ and the $n$-th one is
the
second derivative in the direction $y$ at $x$.\\

\noindent Then, instead of Formula~(\ref{riceminalfa}) we write the
formula

\begin{equation}\label{riceminalfa1}
\begin{split}
\E\big(m_{\alpha}(L,S)\big)=\int_0^{\alpha}&
 du \int_{\psi^{-1}(S)}dw\int_{(S_{\psi (w)}\times \R)^n}
  \E  \left(\big|\det(\tilde{L}''(w))\big|\cdot \chi_{\{
\tilde{L}''(w)\succ 0 \}} \, /\,
  \tilde{L}(w)=u,\right. \\
 & \left. \tilde{L}'(w)=0,\zeta_{\psi(w)}=z\right)
 \cdot
  p_{\tilde{L}(w),\tilde{L}'(w),\zeta_{\psi(w)}}(u,0,z)~dz.
\end{split}
\end{equation}
Formally, Formula~(\ref{riceminalfa}) is obtained from
Formula~(\ref{riceminalfa1}) by integrating in $z$.\\

\noindent To prove the validity of Formula~(\ref{riceminalfa1}) one
could follow exactly the proof of \cite[Formula~6.18]{AW2}  if the
random field $\{L(x,y):(x,y)\in V \}$ were Gaussian. This is not our
case. However, it is in fact a simple function of a Gaussian field,
namely it is a quadratic form in the coordinates of $f$ and its
first derivatives as shown in Formula~(\ref{randomfield}). It is
then easy to show that Formula~(\ref{riceminalfa1}) remains true  as
it is done for the general Rice formulas in \cite[Ch.~6,
Section~1.4]{AW2}. This requires proving: (a) the existence and
regularity of the density $
p_{\tilde{L}(w),\tilde{L}'(w),\zeta_{(\psi (w))}}(u,0,z)$ and (b)
with probability $1$, $0$ is a regular value
of $\tilde{L}'(w) $.\\

\noindent (a) is contained below in the present proof (see Step 4).
As for (b), once the regularity of this density will be established,
it follows
in the same way as \cite[Proposition~6.5 (a)]{AW2}.\\

\noindent So, using exactly the same arguments leading to
Formula~(\ref{riceinicial}) we get:

\begin{equation*}
\aligned \E\big(m_{\alpha}(L,V)\big)=\sigma _V(V)\int_0^{\alpha}
 du \int_{(S_{\psi (0)}\times \R)^n}
 & \E  \left(\big|\det(\tilde{L}''(0))\big|\cdot \chi_{\{
\tilde{L}''(0)\succ 0 \}} \, /\,
  \tilde{L}(0)=u,\right. \\
 & \left. \tilde{L}'(0)=0,\zeta_{\psi(0)}=z\right)
 \cdot
  p_{\tilde{L}(0),\tilde{L}'(0),\zeta_{\psi(0)}}(u,0,z)~dz.
\endaligned
\end{equation*}
Finally, taking into account Inequality~(\ref{distmin}) we can
conclude that:

\begin{equation}\aligned \label{densbound2}
p_{\underline{L}}(u) \leq \sigma _V(V) \int_{(S_{\psi (0)}\times
\R)^n}
 \E  &\left(\big|\det(\tilde{L}''(0))\big|\cdot \chi_{\{
\tilde{L}''(0)\succ 0 \}} \, /\,
  \tilde{L}(0)=u,\right. \\
 & \ \left. \tilde{L}'(0)=0,\zeta_{\psi(0)}=z\right)
 \cdot
  p_{\tilde{L}(0),\tilde{L}'(0),\zeta_{\psi(0)}}(u,0,z)~dz.
\endaligned
\end{equation}

\medskip

\noindent \textbf{Step 3.} For the rest of the proof  we fix the
following orthonormal basis $\mathcal{B}_T$ (given in
(\ref{baseestan})) of the tangent space $T:=T_{e_0,e_1}$:
\begin{equation}\label{basee0e1}
    \mathcal{B}_T=\Big((e_2,0),\ldots,(e_n,0),(0,e_2),\ldots,(0,e_n),
    \frac{1}{\sqrt{2}}(e_1,-e_0)\Big).
\end{equation}
Let us recall that in the right-hand side of
Inequality~(\ref{densbound2}) the values of $\tilde{L}(0),
\tilde{L}'(0), \tilde{L}''(0)$
 are computed using a chart $\psi$ of a neighborhood of $(e_0,e_1)$ such that
 $\psi (0)=(e_0,e_1)$ and the image by $\psi ' $ of the canonical basis of $\R^{2n-1}$ is
  an orthonormal basis of the tangent space $T$, that we set to be
 $\mathcal{B}_T $.\\

\noindent We introduce, for $(x,y)\in V$, the gradient $\nabla
\tilde{L}(x,y)$ which is the orthogonal projection of the free
derivative $L'(x,y)$ onto the tangent space $T_{(x,y)}$ and is
obviously independent of the parametrizations of the manifold $V$.
One can check by means of a direct computation that
$$
\nabla \tilde{L}(e_0,e_1)=\tilde{L}'(0)\big(\psi'(0)\big)^t.
$$
Then, using the change of variables formula for densities and the
fact that $\big(\psi'(0)\big)^t\psi'(0)=I_{2n-1}$, we have:
$$
p_{\tilde{L}(0),\tilde{L}'(0),\zeta_{\psi(0)}}(u,0,z)=p_{L(e_0,e_1),\nabla \tilde{L}(e_0,e_1),\zeta_{(e_0,e_1)}}(u,0,z).
$$

\noindent{\bf{Notation. }}
To simplify notation, from now on we write $f_i$ (resp.
$\partial_kf_i$ and $\partial_{k\ell}f_i$, $0\le k,\ell\le n$) for
$f_i(e_0) $ (resp. $\partial_kf_i(e_0)=\frac{\partial f_i}{\partial
x_k}(e_0)$, $\partial_{k\ell }f_i(e_0)=\frac{\partial^2
f_i}{\partial x_k\partial x_\ell}(e_0)$, $0\le k,\ell\le n$). In the
same spirit we write $L$ for $L(e_0,e_1)=\tilde L(0)$, $\nabla
\tilde{L}$ for $\nabla \tilde{L} (e_0,e_1)$ and  $L''$ for
$L''(e_0,e_1)$.  Finally we write $\zeta$ for
 $\zeta(e_0,e_1)$ and $S$ for $S_{(e_0,e_1)}$.\\

\noindent Under this notation, Inequality~(\ref{densbound2}) becomes:
\begin{equation} \label{densbound3}
p_{\underline{L}}(u) \leq \sigma _V(V) \int_{(S\times \R)^n}
 \E  \left(\big|\det(\tilde{L}'')\big|\cdot \chi_{\{
\tilde{L}''\succ 0 \}}\Big/ {L}=u, \nabla \tilde{L}=0,\zeta=z\right)
p_{{L},\nabla \tilde{L},\zeta}(u,0,z)~dz.
\end{equation}
According to the definition of $L(x,y)$ in~(\ref{randomfield}) we
have
\begin{equation}\label{ele}
   L=\sum_{i=1}^n\frac{1}{d_i}(\partial_1f_i)^2+\sum_{i=1}^nf_i^2,
\end{equation}
and, from Definition~(\ref{zetaxv}),
\begin{equation}\label{zetaeceroe1}
  \zeta:= \zeta_{e_0,e_1}=\big((\partial_2f_i,\dots,\partial_nf_i,
     \partial_{11}f_i ) , 1\le i\le n\big)\in\R^{n^2}.
\end{equation}
We also set
$[\nabla \tilde{L}]_{\mathcal{B}_T}:=(\xi_2,\ldots,\xi_n,\eta_2,\ldots,
\eta_n,\varrho)$ for the coordinates of  the gradient
$\nabla \tilde{L}$ in the basis $\mathcal{B}_T$.\\

\noindent Using that the (free) partial derivatives of $L$ at
$(e_0,e_1)$   are given  by
\begin{equation*}\label{derprimlibre}\aligned
 &\frac{\partial L}{\partial x_{k}}(e_0,e_1)=
  \sum_{i=1}^n\frac{2}{d_i} (\partial_{k1}f_i)(\partial _1f_i)
    +\sum_{i=1}^n2f_i(\partial _{k}f_i) \quad
     \mbox{for\ } \ 0\le k\le n\\
 &\frac{\partial L}{\partial y_\ell}(e_0,e_1)
  =\sum_{i=1}^n\frac{2}{d_i}(\partial _1f_i)(\partial_\ell f_i)
     \quad \mbox{for\ } \ 0\le \ell\le n,
\endaligned
\end{equation*}
we obtain
\begin{equation}\aligned \label{derVcoord}
   \xi _{j}&=\langle L'(e_0,e_1),(e_{j},0)\rangle
   =2\sum_{i=1}^n\frac{1}{d_i}(\partial _{1j}f_i)(\partial_1f_i)
    +2\sum_{i=1}^nf_i(\partial_{j}f_i), \ \  2\le j\le n,\\
    \eta_{j}&=\langle L'(e_0,e_1),(0,e_{j})\rangle
  =2\sum_{i=1}^n\frac {1}{d_i}(\partial_1 f_i)(\partial_{j}f_i), \
     2\le j\le n,\\
    \varrho&=\langle L'(e_0,e_1),2^{-1/2}(e_1,-e_0)\rangle\\
  &=\sqrt{2}\Big[\sum_{i=1}^n\frac{1}{d_i}(\partial_1f_i)
    (\partial_{11}f_i)+\sum_{i=1}^nf_i(\partial_{1}f_i)\Big]
   -\sqrt{2}\sum_{i=1}^n\frac{1}{d_i}(\partial_{0}f_i)
    (\partial_{1}f_i)\\
  &=\sqrt{2}\sum_{i=1}^n\frac{1}{d_i}(\partial_1f_i)
        (\partial_{11}f_i).
\endaligned
\end{equation}
Here, $\langle\ ,\ \rangle$ denotes the usual inner product in
$\R^{n+1}\times \R^{n+1}$ and the last equality in (\ref{derVcoord})
follows from the equalities $\partial_0f_i=d_if_i$ for $1\leq i\leq n$
which are easily verified.

\bigskip

\noindent\textbf{Step 4}. In this step we focus on the term
$p_{L,\nabla\tilde{L},\zeta}(u,0,z)$ of \eqref{densbound3}. To this aim
we factor this density as
\begin{eqnarray}\label{factdensidad}
   p_{L ,\nabla \tilde{L} ,\zeta}(u,0,z)& = &
   q_{L,\nabla \tilde{L}/\zeta=z}(u,0)\,\cdot\,p_{\zeta}(z)
\end{eqnarray}
where $q_{L,\nabla \tilde{L}/\zeta=z}(u,0)$ denotes conditional density. \\

\noindent To study the two terms in the right-hand side
of~(\ref{factdensidad}), we need a  lemma containing the ingredients to compute the
distributions and conditional expectations appearing in our proof.

 \begin{lemma}\label{lemcovar}
Let $f\in \R[X_0,\dots,X_n]$ be a homogeneous random polynomial of
degree $d$. Assume that $f$ follows the Shub-Smale model for the
probability law of the coefficients, i.e. the coefficients of the
polynomial $f=\sum_{|j|=d} a_j X^j$ are independent, Gaussian,
centered random variables  with variances
\[
\Var(a_j)=
       {  d \choose
              j }.
\]
Then
\begin{itemize}
\item For $x,y\in \R^{n+1}$, the covariances satisfy
$$
   \E \left(f(x)f(y) \right)=\langle x,y \rangle ^{d}~~~
   \forall \,x,y\in\R^{n+1},
$$
where $\langle \ ,\ \rangle$ is the usual inner product in
$\R^{n+1}$.
\end{itemize}
Moreover, if $e_0:=(1,0,\dots,0)$ is the first vector of the
canonical basis of $\R^{n+1}$ and we write  $f$ (resp. $\partial_kf$
and $\partial_{k\ell}f$, $0\le k,\ell\le n$) for $f(e_0) $ (resp.
$\partial_kf(e_0)=\frac{\partial f}{\partial x_k}(e_0)$,
$\partial_{k\ell }f(e_0)=\frac{\partial^2 f}{\partial x_k\partial
x_\ell}(e_0)$, $0\le k,\ell\le n$), we get the following
covariances:
\begin{itemize}
\item $\E \left( f \partial_kf \right) =\delta_{k0}d$ \ for \  $0\le k\le n$.
\item $\E \left( (\partial_kf)(\partial_{k'}f)\right) =\delta_{kk'}[d+\delta_{k0}d(d-1)]$  \ for \  $0\le k,k'\le n$.
\item $\E \left( f(\partial_{k\ell}f)\right) =\delta_{k\ell}\delta_{k0}d(d-1)$  \ for \  $0\le k,\ell\le n$.
\item $\E \left( (\partial_{k\ell}f)(\partial_{k'}f)\right) =d(d-1)\big[
(d-2)\delta_{\ell0}\delta_{k0}\delta_{k'0}+\delta_{k0}\delta_{k'\ell}+\delta_{\ell0}\delta_{kk'}\big]$
 \ for \  $0\le k,k',\ell\le n$.
\item $\E \left( (\partial_{k\ell}f)(\partial_{k'\ell'}f)\right)
=d(d-1)\Big\{
(d-2)(d-3)\delta_{k0}\delta_{\ell0}\delta_{k'0}\delta_{\ell'0}+(d-2)
\big[\delta_{k0}\delta_{k'0}\delta_{\ell\ell'}+\delta_{k'0}\delta_{\ell0}\delta_{k\ell'}+
\delta_{k0}\delta_{\ell'0}\delta_{k'\ell}+\delta_{\ell0}\delta_{\ell'0}\delta_{kk'}\big]+
\delta_{kk'}\delta_{\ell\ell'}+\delta_{k\ell'}\delta_{k'\ell}\Big\}$ \ for \ $0\le k,k',\ell,\ell'\le n.\qquad \diamondsuit$
\end{itemize}
\end{lemma}

\medskip
\noindent We proceed with the  study of the two terms in the right-hand side
of~(\ref{factdensidad}).\\

\noindent {\em Computation of  $p_{\zeta }(z)$}:
By Lemma~\ref{lemcovar}, the $n^2$
coordinates of $\zeta $ in \eqref{zetaeceroe1} are independent Gaussian centered random
variables satisfying that
$\Var(\partial _{k}f_i)=d_i$
and $\Var(\partial _{11}f_i)=2d_i(d_{i}-1)$   for $1\le i\le n$ and $2\le k\le n$.\\
Although we are not going to use the exact expression in the sequel,
we can immediately deduce for $z=\left((z_{i2},\dots
,z_{in},z_{i11}), 1\le i\le n \right)$ that
\[
   p_{\zeta }(z)\,=\,
  \frac{1}{(2\pi)^{n^2/2}}\frac{1}{\prod_{i=1}^nd_i^{(n-1)/2}
  \prod_{i=1}^n(2d_i(d_i-1))^{1/2}}\exp \left(
  -\frac{1}{2}\sum_{i=1}^n\left( \sum_{j=2}^n
  \frac{z_{ij}^2}{d_i}+\frac{z_{i11}^2}{2d_i(d_i-1)}\right)\right).
\]

\medskip \noindent
 {\em Computation of
$q_{L,\nabla \tilde{L}/\zeta=z}(0)$}: We factor it  as follows:
$$
   q_{L,\nabla \tilde{L}/\zeta=z}(u,0)= q_{L /\nabla \tilde{L} =0,\zeta =z}(u)
  \cdot q_{\nabla \tilde{L} /\zeta=z}(0).
$$
Remembering that $\big(\nabla
\tilde{L}\big)_{\mathcal{B}_T}:=(\xi_2,\dots,\xi_n,\eta_2,\dots,
\eta_n,\varrho)$, we can write $q_{\nabla \tilde{L}/\zeta=z}(0)$ as
$$
  q_{\nabla \tilde{L}/\zeta=z}(0)=q_{(\xi_2,\ldots,\xi_n)/
    (\eta_2,\ldots,\eta _n, \varrho)=0,~\zeta=z}(0)\cdot
  q_{(\eta_2,\ldots,\eta _n, \varrho)/\zeta=z}(0).
$$

\noindent First we  compute $q_{(\eta _2,\ldots,\eta
_n,\varrho)/\zeta =z}(0)$. The condition $\zeta =z$ says that for
$1\le i\le n$ and $2\le j\le n$, $\partial_j f_i=z_{ij}$ and
$\partial_{11}f_i=z_{i11}$. Therefore, from Identities
(\ref{derVcoord}), we have
\begin{equation}\label{Az}
   \left(\begin{array}{c}\eta_2\\ \vdots\\
   \eta_n\\ \varrho\end{array}\right)
  = A(z)\ \left(\begin{array}{c}\frac{\partial_1f_1}{\sqrt{d_1}}\\
     \vdots\\ \frac{\partial_1f_n}{\sqrt{d_n}}\end{array}\right),\quad
    \mbox{where} \quad
    A(z)=\stackrel{\longleftarrow \quad \bajo n \quad \longrightarrow}{ \left( \begin{array}{ccc}
    \frac{2}{\sqrt{d_1}}z_{12} & \dots &\frac{2}{\sqrt{d_n}}z_{n2}\\
    \vdots& & \vdots \\
    \frac{2}{\sqrt{d_1}} z_{1n} & \dots
&\frac{2}{\sqrt{d_n}} z_{nn} \\[4pt]
\hline
\alto\frac{\sqrt{2}}{\sqrt{d_1}}z_{111}& \dots &
  \frac{\sqrt{2}}{\sqrt{d_n}}z_{n11}
\end{array}\right) }
\begin{array}{c}
\scriptstyle{\uparrow}\\
\scriptstyle{n-1}\\\scriptstyle{\downarrow}\\
[6pt]\\
\scriptstyle{1}
\end{array}
\end{equation}
is non-singular for almost every $z\in\R^{n^2}$.
Applying again Lemma~\ref{lemcovar}, $\partial_1f_i/\sqrt{d_i}$,
$1\le i\le n$,  are independent standard normal random variables
that are independent from $\zeta$. By the change of variables
formula, we get
$$
  q_{(\eta _2,\ldots,\eta _n,
 \varrho)/\zeta=z}(0)=\frac{1}{(2\pi)^{n/2}}\cdot \frac{1}{|\det A(z)|}.
$$

\noindent Now we compute
$q_{(\xi_2,\ldots,\xi_n)/(\eta_2,\ldots,\eta _n,
\varrho)=0,\zeta=z}(0)$. Since $A(z)$ is non-singular for almost
every $z$, the condition $\eta _2=\ldots=\eta _n = \varrho=0 $
implies $\partial _1f_i=0$ for $1\le i\le n$. Therefore, from
Identities (\ref{derVcoord}) and since $\zeta=z $, we have
$$
  \left(\begin{array}{c}\xi _{2}\\ \vdots \\ \xi_n\end{array}\right)
  =2\, B(z) \,
  \left(\begin{array}{c}f_1\\ \vdots \\f_n\end{array}\right),
  \ \  \mbox{where } \ \
  B(z)= \stackrel{\longleftarrow \quad \bajo n \quad \longrightarrow }{\left(
     \begin{array}{ccc}
    z_{12} & \dots &z_{n2}\\
 \vdots& & \vdots\\
   z_{1n} & \dots & z_{nn}
\end{array}\right)}{\begin{array}{c}\uparrow \\ \scriptstyle n-1 \\ \downarrow\end{array} }.
$$
Again,  $f_1,\dots,f_n$ are independent
standard normal variables independent from
$\left( \eta_2,\ldots,\eta _n, \varrho,\zeta \right)$ and thus
$$
  q_{(\xi_2,\ldots,\xi_n)/(\eta _2,\ldots,\eta _n, \varrho)=0,~\zeta
  =z}(0)=\frac{1}{(2\pi)^{(n-1)/2}}\cdot \frac{1}{2^{n-1} (\det
  (B(z)B(z)^t))^{1/2}},
$$
where $B(z)^t$ denotes the transpose of the matrix $B(z)$.\\
We therefore obtain
\begin{eqnarray*}
  q_{\nabla \tilde{L}/\zeta
 =z}(0) & =&q_{(\xi_2,\ldots,\xi_n)/(\eta _2,\ldots,\eta _n,
 \varrho)=0,~\zeta=z}(0)\cdot q_{(\eta _2,\ldots,\eta _n,
 \varrho)/\zeta=z}(0)\\
 &=&\frac{1}{(2\pi )^{n-\frac{1}{2}}2^{n-1}|\det A(z)|(\det
 (B(z)B(z)^t))^{1/2}}.
\end{eqnarray*}
\medskip

\noindent Finally we compute  $q_{L /\nabla \tilde{L} =0,\zeta=z}(u)$. The
conditions $\nabla \tilde{L} =0$ and $\zeta =z$ imply by \eqref{Az} and \eqref{derVcoord}
 that  $\partial _1f_i=0$ for $1\le i\le n$   and $\sum_{i=1}^nf_iz_{ij }=0$ for $2\le j\le n$ for almost every $z$
 . Plugging the former into
\eqref{ele} we get
$$
L= \sum_{i=1}^nf_i^2,
$$
and the latter says that
 the vector $(f_1,\ldots,f_n)$  is orthogonal to the
$(n-1)$-dimensional subspace $S$  spanned by the $n-1$ vectors
$(z_{1j},\dots,z_{nj}),\ 2\le j\le n$. This shows that $f_1^2+\cdots +f_n^2$, the square of the distance of $(f_1,\dots,f_n)$ to $S$,
has the
$\chi^2_1$-distribution, since
the property of being a vector of independent standard normal
variables is independent of the choice of the orthonormal basis.
 So, for $u>0$,
$$
  q_{L /\nabla \tilde{L} =0,~\zeta  =z}(u)=\frac{e^{-u/2}}{\sqrt{2\pi u }}.
$$
We therefore obtain
\begin{equation*}\label{denscondparcial}\aligned
  q_{L,\nabla \tilde{L}/\zeta=z}(u,0)&= \ q_{L /\nabla \tilde{L} =0,\zeta=z}(u)
  \cdot q_{\nabla \tilde{L} /\zeta=z}(0)\\
  &=\ \frac{e^{-u/2}}{(2\pi)^{n}2^{n-1}|\det\big( A(z)\big)|(\det
    \big(B(z)B(z)^t)\big)^{1/2}\,\sqrt{u}}.
\endaligned
\end{equation*}
Plugging this expression into Identity~(\ref{factdensidad}) we
obtain
\begin{eqnarray}\label{factdensidad2}
   p_{L ,\nabla \tilde{L} ,\zeta}(u,0,z)& = &
   \frac{e^{-u/2}}{(2\pi)^{n}2^{n-1}|\det \big(A(z)\big)|(\det
    \big(B(z)B(z)^t)\big)^{1/2}\,\sqrt{u}}\,\cdot\,p_{\zeta}(z).
\end{eqnarray}

\bigskip

\noindent\textbf{Step 5}.
In this step we focus on  the conditional expectation
\begin{equation}\label{Lsegunda}
   \E\left(\big|\det(\tilde{L}'')\big| \cdot
   \chi_{\{\tilde{L}''\succ 0\}} \Big/ L=u,\nabla \tilde{L}=0,\zeta=z \right)
\end{equation}
in the integrand of (\ref{densbound3}). We obtain the following expression for $\tilde{L}''$ under the stated conditions.\\

\begin{lemma}\label{lemmaM}
Let $M$ be the symmetric block-matrix $ \R^{(2n-1)\times
(2n-1)}$ of the linear operator $\tilde{L}''$,  under the conditions
$L =u, \nabla \tilde L =0 $ and $\xi =z$.  Let $f^*$ be any solution of the system $\sum_{i=1}^n f_iz_{ij}=0$, $2\le j\le n$,
and $\sum_{i=1}^n f_i^2=u$. Then
$$
    M=\stackrel{\bajo n-1\qquad n-1\qquad 1\ }{\left(
    \begin{array}{c|c|c}
    M_{\sigma \sigma}
   & M_{\sigma \tau}&M_{\sigma \theta }\\[2pt]
   \hline
  \alto M_{\tau \sigma}&M_{\tau \tau} & M_{\tau \theta}\\[2pt]
  \hline
  \alto M_{\theta \sigma} & M_{\theta \tau} & M_{\theta \theta}
  \end{array}\right) }
  \begin{array}{l}
  \alto{\scriptstyle n-1} \\
  \alto{\scriptstyle n-1} \\
  \alto{\scriptstyle 1}
  \end{array}
$$
where
\begin{equation*}\aligned
  & (M_{\sigma\sigma})_{jj} \ =  \ 2\sum_{i=1}^n
  \left(\frac{1}{d_i}(\partial_{1j}f_i)^2+z_{ij }^2+f^*_i
  (\partial_{jj}f_i)-d_i{f^*_i}^2\right) \quad
  \mbox{for} \  2\le j\le n,\\
  & (M_{\sigma\sigma})_{jk} \  = \ 2\sum_{i=1}^n
  \left(\frac{1}{d_i}(\partial_{1j}f_i)(\partial_{1k}f_i)
  +z_{ij }z_{ik}+f^*_i(\partial_{jk}f_i) \right) \quad
  \mbox{for} \ 2\le j\ne k \le n,\\
  & (M_{\sigma\tau})_{jk} \ = \ 2\sum_{i=1}^n
  \frac{1}{d_i}(\partial_{1j}f_i)z_{ik } \quad
  \mbox{for} \  2\le j,k\le n ,\\
 & (M_{\sigma\theta})_{j1} \ = \ \sqrt{2}\sum_{i=1}^n
   \frac{1}{d_i}(\partial_{1j}f_i) z_{i11} \quad
   \mbox{for} \  2\le j\le n,\\
  & (M_{\tau\tau})_{jk} \ = \ 2\sum_{i=1}^n
  \frac{1}{d_i}z_{ij}z_{ik } \quad \mbox{for} \  2\le j, k\le n,\\
  & (M_{\tau\theta})_{j1} \ = \ \sqrt{2}\sum_{i=1}^n
  \frac{1}{d_i}z_{i11}z_{ij }  \quad \mbox{for} \  2\le j\le n,\\
  & M_{\theta\theta}\ = \  \sum_{i=1}^n
  \left(\frac{1}{d_i}z_{i11}^2-f^*_iz_{i11}\right). \endaligned
\end{equation*}
\end{lemma}
\begin{proof}
The hypotheses imply that for
almost every $z$, one has $\partial_1f_i=0$ for $1\le i\le n$,
$\sum_{i=1}^n f_iz_{ij}=0$ for $2\le j\le n$
and $\sum_{i=1}^n f_i^2=u$. The last two conditions
give a system of $n$ equations and
$n$ unknowns  with  exactly two solutions
$f^*=(f^*_1,\ldots,f^*_n)$ and $-f^*$ for almost every $z$ and
$u>0$.
Moreover the symmetry of the
Gaussian distribution implies that the law of the coordinates of the matrix
$M$ does not change under the stated
conditions when replacing $f_1,\ldots,f_n$ by either one of these solutions.
The formulas are then a consequence of Corollary \ref{matrizM} of Section
\ref{auxlemmas} (here we use   that $\partial_0 f_i =d_if_i$ and skip the details).
\end{proof}

\eproof

\bigskip
\noindent
For $z$ fixed, the only random variables that appear in the elements of $M$
are the second partial derivatives $\partial_{jk}f_i$, $2\le j, k\le n$ and $\partial_{1j}f_i$, $2\le j\le n, 1\le i\le n$.
Therefore, we are in condition to   apply the following result which  gets rid of conditioning in \eqref{Lsegunda}.

\begin{lemma}\label{condincond}
Let $X=(X_{ij})_{1\le i\le p,1\le j\le q}$ be a real random matrix and $Y=(Y_1,...,Y_q)^t,~Z=(Z_1,...,Z_p)^t$ be real random vectors.
Assume that $X,~Y,~Z$ are independent, the distributions of $X,$ $Y$ and $Z$ have bounded continuous densities, respectively in $\R^{p\times q},~\R^{q},~\R^{p}$ and that $p_Y(.)$ and $p_Z(.)$ do not vanish.
Let $g:\R^{p\times q}\rightarrow \R$ be continuous, such that $\E ( |g(X)|)<+\infty.$

\noindent
Then, for any $u\in \R^p$,

$$
\E \big( g(X)\, /\, XY+Z=u, Y=0\big)=\E \big( g(X)\big).\qquad \diamondsuit
$$

\end{lemma}
\noindent  The heuristic meaning of the previous lemma is that \emph{if we know that} $Y=0$, then $XY+Z$ does not give information on the distribution of $X$.

\medskip
\noindent For $X=\Big(\frac{1}{d_i}\partial_{1j}f_i\Big)_{2\le j\le n,1\le i\le n}\in \R^{(n-1)\times n}$ and  $
Y=
\big( \partial_1f_1,\dots,\partial_1f_n \big)^t
$ in the previous lemma
we obtain that
\begin{equation}\label{condexp2}
   \E \left(\big|\det(\tilde{L}'')\big|
    \cdot \chi_{\{\tilde{L}''\succ 0 \}}\Big/
    L=u,\nabla \tilde{L}=0,\zeta=z\right) \
   = \  \E \left(\big|\det(M)\big| \cdot \chi_{\{M\succ 0 \}}\right).
\end{equation}

 \noindent  We now consider $\E \left(\big|\det(M)\big| \cdot \chi_{\{M\succ 0 \}}\right)$. We observe that it is now an unconditional expectation.
We will bound it in terms of $u$ and $z$.
We begin by writing the matrix $M$ in a form that will be
useful for our computations. \\

\noindent
{\bf{Notation} }
To simplify notation, from now on  we
simply write $A$ and $B$ for the matrices $A(z)$ and $B(z)$ of Step
4.\\

\noindent  We first observe that $$ M_{\sigma\sigma}=
VV^t+2BB^t+W-\mu I_{n-1}$$ where
$$
  V:= \stackrel{\longleftarrow\quad \bajo n\quad \longrightarrow}{\left(
    \begin{array}{ccc}
  \frac{\sqrt 2}{\sqrt{d_1}}\partial_{12}f_1 & \dots
  &\frac{\sqrt 2}{\sqrt{d_n}}\partial_{12}f_n\\
  \vdots& & \vdots\\
  \frac{\sqrt 2}{\sqrt{d_1}}\partial_{1n}f_1 & \dots
 &\frac{\sqrt 2}{\sqrt{d_n}}\partial_{nn}f_n
 \end{array}\right)}\begin{array}{c}\uparrow\\ {\scriptstyle{n-1}}\\\downarrow \end{array},\quad
 W:=\stackrel{\longleftarrow\quad \bajo n-1\quad \longrightarrow}{\left(
  \begin{array}{ccc}
  2\sum_{i=1}^nf_i^*\partial_{22}f_i & \dots
 &2\sum_{i=1}^nf_i^*\partial_{2n}f_i\\
 \vdots& & \vdots\\
 2\sum_{i=1}^nf_i^*\partial_{n2}f_i & \dots
 &2\sum_{i=1}^nf_i^*\partial_{nn}f_i
\end{array}\right)}
\begin{array}{c}\uparrow\\ {\scriptstyle{n-1}}\\\downarrow \end{array}
$$
and $$ \mu : = 2\sum_{i=1}^n d_i{f^*_i}^2.
$$
Also, introducing for $1\le i\le n$ and $2\le j\le n$,
$$
  \tilde{z}_{ij}:=\frac{2}{\sqrt{d_i}}z_{ij},
  \quad \tilde{z}_{j}:=\big(\tilde z_{1j}, \dots,   \tilde z_{nj}\big),
  \quad \hat{B}=\hat{B}(z):
  = \left(\begin{array}{c}
   \tilde z_2 \\ \vdots \\ \tilde z_n
   \end{array}\right)
  = \stackrel{\longleftarrow\quad \bajo n\quad \longrightarrow}{
   \left(\begin{array}{ccc}
   \tilde z_{12}&\dots & \tilde z_{n2} \\
   \vdots& &\vdots \\
   \tilde z_{1n}&\dots & \tilde z_{nn}
   \end{array}\right)}
\begin{array}{c}\uparrow\\ {\scriptstyle{n-1}}\\\downarrow \end{array}$$
and
$$
   \tilde{z}_{i11}:=\frac{2}{\sqrt{d_i}}z_{i11}, \
   \tilde{f}_i:=\sqrt{d_i}f_i^* \quad\mbox{and} \quad
   \tilde{z}_{11}:=\big( \tilde z_{111}, \dots,  \tilde z_{n11}\big) ,
   \ \tilde{f}:=\big(\tilde{f}_1,\dots,\tilde{f}_n\big)
$$
so that
$$
  A=\stackrel{\bajo n}{
  \left(\begin{array}{c}
  \hat B \\[2pt] \hline
  \alto \frac{1}{\sqrt 2}\tilde z_{11}
  \end{array}\right)}
  \begin{array}{c}
  {\scriptstyle n-1}\\[2pt]
  {\alto \scriptstyle 1}
  \end{array},
$$
we get
$$
   M_{\sigma \tau}= \frac{1}{\sqrt{2}}V \hat{B}^t, \  M_{\sigma \theta }
   =  \frac{1}{2} V \tilde{z}_{11}^t,
   \  M_{\tau \tau} =\frac{1}{2}\hat{B}\hat{B}^t , \ M_{\tau\theta}
  =\frac{\sqrt{2}}{4}\hat{B}\tilde{z}_{11}^t \ \mbox{and}\
  M_{\theta \theta}= \frac{1}{4}\tilde{z}_{11}\tilde{z}_{11}^t
  -\frac{1}{2}\tilde{z}_{11}\tilde{f}^t.
$$
Therefore
$$
  M=\stackrel{\bajo n-1 \hspace*{80pt} n-1 \hspace*{50pt} 1}
       {\left(\begin{array}{c|c|c}
       VV^t+2BB^t+W-\mu I_{n-1}
      & \frac{1}{\sqrt{2}}V \hat{B}^t &\frac{1}{2} V \tilde{z}_{11}^t \\ [3pt]
      \hline
      \alto \frac{1}{\sqrt{2}}\hat{B}V^t&
      \frac{1}{2}\hat{B}\hat{B}^t&
      \frac{\sqrt{2}}{4}\hat{B}\tilde{z}_{11}^t \\ [3pt]
      \hline
      \alto \frac{1}{2}\tilde{z}_{11}V^t&
      \frac{\sqrt{2}}{4}\tilde{z}_{11}\hat{B}^t  &
      \frac{1}{4}\tilde{z}_{11}\tilde{z}_{11}^t -\frac{1}{2}\tilde{z}_{11}
      \tilde{f}^t
      \end{array}\right)}
  \begin{array}{c}
  {\scriptstyle n-1} \\ [4pt]
  \alto {\scriptstyle n-1} \\ [4pt]
  \alto {\scriptstyle 1}
 \end{array}.
$$

\smallskip \noindent
The coefficients of the matrix $W$ appearing in the first block are the centered Gaussian random variables
$\{2\sum_{i=1}^nf_i^*\partial_{jk}f_i:2 \leq j \leq k \leq n \}$
which are independent. Applying Lemma \ref{lemcovar}, we obtain
\begin{equation}\label{varw}
\aligned
&\sigma^2:=\Var(2\sum_{i=1}^nf_i^*\partial_{jk}f_i)
  =4\sum_{i=1}^nd_i (d_i-1){f_i^*}^2\ \leq \ 4\bD(\bD-1)u \
 \quad \text{for}\quad j\neq k,\\
&
\Var(2\sum_{i=1}^nf_i^*\partial_{jj}f_i)=8\sum_{i=1}^nd_i
(d_i-1){f_i^*}^2=2\sigma^2.
\endaligned
\end{equation}
As a consequence, dividing each coefficient of $W$ by $\sigma
\sqrt{n-1}$, one can write the matrix $W$ in the
form:
$$
W=\sigma \sqrt{n-1} \,G
$$
where $G$ is a real random symmetric matrix with entries $a_{ij}$ which
are independent Gaussian centered  satisfying that
$\Var (a_{ij})=1/n$ for $i \neq j$ and $\Var (a_{ij})=2/n$ for $i =
j$.\\

\noindent We continue now with the  bound for  $\E \left(\big|\det(M)\big| \cdot
\chi_{\{M\succ 0 \}}\right)$. The
randomness for this expectation lies in the matrices $V$ and $W$,
which are stochastically independent
by Lemma \ref{lemcovar}.\\

\noindent Denote by $\bar{\lambda}$ the maximum between 0 and the largest
eigenvalue of the matrix $G$. Using
the independence of $V$ and $W$, and the fact that
the determinant of a positive semidefinite matrix is
an increasing function of the diagonal values, we get
\begin{equation}\label{inM1}
\E \left(\big|\det(M)\big| \cdot \chi_{\{M\succ 0 \}}\right) \ \leq
\ \E \left(\big|\det(M_1)\big| \cdot \chi_{\{M_1\succ 0 \}}\right)
\end{equation}
where $M_1$ is given by:
$$
  M_1=\stackrel{\bajo n-1 \hspace*{80pt} n-1 \hspace*{50pt} 1}
   {\left(
   \begin{array}{c|c|c}
   VV^t+2BB^t+\sigma\, \sqrt{n} \,\bar{\lambda} \, I_{n-1}
    & \frac{1}{\sqrt{2}}V \hat{B}^t &\frac{1}{2} V \tilde{z}_{11}^t \\ [3pt]
    \hline
   \alto \frac{1}{\sqrt{2}}\hat{B}V^t&
   \frac{1}{2}\hat{B}\hat{B}^t&
  \frac{\sqrt{2}}{4}\hat{B}\tilde{z}_{11}^t \\ [3pt]
  \hline
  \alto
  \frac{1}{2}\tilde{z}_{11}V^t& \frac{\sqrt{2}}{4}\tilde{z}_{11}\hat{B}^t  &
  \frac{1}{4}\tilde{z}_{11}\tilde{z}_{11}^t -\frac{1}{2}\tilde{z}_{11}\tilde{f}^t
  \end{array}\right)}
  \begin{array}{c}
  {\scriptstyle n-1} \\ [4pt]
  \alto {\scriptstyle n-1} \\ [4pt]
  \alto {\scriptstyle 1}
  \end{array}.
$$
We note that
\begin{equation}\label{inM1_1}
  \det(M_1)=\det (M_2)-\frac{1}{2}\tilde{z}_{11}\tilde{f}^t \det(M_0).
\end{equation}
where
$$
M_0=\stackrel{\bajo \hspace*{40pt} n-1 \hspace*{80pt} n-1}
         {\left(\begin{array}{c|c}
          VV^t+2BB^t+\sigma\, \sqrt{n} \,\bar{\lambda} \, I_{n-1}
          & \frac{1}{\sqrt{2}}V \hat{B}^t \\ [3pt]
          \hline
         \alto \frac{1}{\sqrt{2}}\hat{B}V^t&
         \frac{1}{2}\hat{B}\hat{B}^t
  \end{array}\right)}
  \begin{array}{c}
  {\scriptstyle n-1} \\ [4pt]
  \alto {\scriptstyle n-1}
  \end{array}.
$$
and
$$
  M_2=\stackrel{\bajo \hspace*{30pt} n-1 \hspace*{70pt} n-1
            \hspace*{40pt} 1}
        {\left(\begin{array}{c|c|c}
        VV^t+2BB^t+\sigma\, \sqrt{n} \,\bar{\lambda} \, I_{n-1}
        & \frac{1}{\sqrt{2}}V \hat{B}^t &\frac{1}{2} V \tilde{z}_{11}^t \\ [3pt]
        \hline
        \alto \frac{1}{\sqrt{2}}\hat{B}V^t& \frac{1}{2}\hat{B}\hat{B}^t&
         \frac{\sqrt{2}}{4}\hat{B}\tilde{z}_{11}^t \\ [3pt]
         \hline
         \alto \frac{1}{2}\tilde{z}_{11}V^t&
        \frac{\sqrt{2}}{4}\tilde{z}_{11}\hat{B}^t  &
        \frac{1}{4}\tilde{z}_{11}\tilde{z}_{11}^t
        \end{array}\right)}
  \begin{array}{c}
  {\scriptstyle n-1} \\ [4pt]
  \alto {\scriptstyle n-1} \\ [4pt]
  \alto {\scriptstyle 1}
  \end{array}.
$$
Observe that $M_0$ and $M_2$ can be
written as
$$
   M_0=N_0\,N_0^t \quad \mbox{and} \quad M_2=N_2\,N_2^t$$ where
$$
  N_0:=\stackrel{\bajo n \hspace*{30pt} n
            \hspace*{40pt} n-1\hspace*{20pt} }
    {\left(\begin{array}{c|c|c}
    V &{\sqrt{2}}B &(\sigma\, \sqrt{n} \,\bar{\lambda})^{1/2} \, I_{n-1}\\ [3pt]
    \hline
    \alto \frac{1}{\sqrt{2}}\hat{B}& 0& 0
        \end{array}\right)}
  \begin{array}{c}
  {\scriptstyle n-1} \\ [4pt]
  \alto {\scriptstyle n-1}
  \end{array}
$$
and
$$
N_2:=\stackrel{\bajo n \hspace*{30pt} n
            \hspace*{40pt} n-1\hspace*{20pt} }
        {\left(\begin{array}{c|c|c}
      V & {\sqrt{2}}B &(\sigma\, \sqrt{n} \,\bar{\lambda})^{1/2} \, I_{n-1}
       \\ [3pt]
      \hline
      \alto \frac{1}{\sqrt{2}}\hat{B}& 0& 0 \\ [4pt]
      \hline
      \alto \frac{1}{2}\tilde{z}_{11} &0 & 0
      \end{array}\right)}
  \begin{array}{c}
  {\scriptstyle n-1} \\ [4pt]
  \alto {\scriptstyle n-1} \\ [4pt]
  \alto {\scriptstyle 1}
  \end{array}.
$$
Therefore  they are both positive semidefinite. Moreover
$\det(M_2)$ is the square of the $(2n-1)$-volume of the
parallelotope generated by the $2n-1$ rows of $N_2$. This
volume equals the distance from the last row to the subspace
generated by the rows of $N_0$ times the volume of the parallolotope
defined by these $2n-2$ rows. The distance from the last row to the
subspace generated by the rows of $N_0$ is bounded by the distance
to the smaller subspace generated by the $n-1$ rows of the matrix
$$
\left(
  \begin{array}{c|c|c}
    \frac{1}{\sqrt{2}} \hat{B} & 0 & 0 \\
  \end{array}
\right),
$$
which is clearly equal to
$$
  \dist \left(\frac{1}{{2}}\tilde{z}_{11},\tilde{S}\right)
$$
where  $\tilde S:=\mbox{span}(\tilde z_2,\dots,\tilde z_n)\subset
\R^{n}$. Now we recall that $(f^*_1,\dots, f^*_n)$ satisfies the
conditions $\sum_{i=1}^n f^*_iz_{ij}=0$, $2\le j\le n$, which
implies
$$
  \langle \tilde{f},\tilde{z}_{j}\rangle=2\sum _{i=1}^nf_i^*z_{ij}=0,
  \ 2\le j\le n.
$$
This means that $\tilde{f}$ is orthogonal to $\tilde{S}$ so that
$$
  \dist \left( \frac{1}{{2}}\tilde{z}_{11},\tilde{S}\right)
  =\frac{1}{{2}}\,\Big|\left\langle
  \frac{\tilde{f}}{\|\tilde{f}\|},\tilde{z}_{11}\right\rangle
  \Big|.
$$
Therefore
\begin{equation}\label{detm2}\det(M_2)\le \frac{1}{4}\,\Big|\left\langle
  \frac{\tilde{f}}{\|\tilde{f}\|},\tilde{z}_{11}\right\rangle \Big|^2 \det(N_0)^2\ = \ \frac{1}{4}\,\Big|\left\langle
  \frac{\tilde{f}}{\|\tilde{f}\|},\tilde{z}_{11}\right\rangle \Big|^2
  \det(M_0).\end{equation}
Using this equality to replace $\det(M_2)$ in~\eqref{inM1_1},
we have that
\begin{equation*}
|\det (M_1)| \leq \frac{1}{2}\, \Big( \frac{1}{2}\Big|\left\langle
\frac{\tilde{f}}{\|\tilde{f}\|},\tilde{z}_{11}\right\rangle
\Big|^2+\big|\langle \tilde{f},\tilde{z}_{11}\rangle
\big|\Big)\,\det(M_0),
\end{equation*}
and therefore, since $M_0$ is positive semidefinite
\begin{equation}\label{ineq1}
\E\big(|\det (M_1)|\cdot \chi_{\{M_1\succ 0 \}}\big) \leq
\frac{1}{2}\Big( \frac{1}{2}\Big|\left\langle
\frac{\tilde{f}}{\|\tilde{f}\|},\tilde{z}_{11}\right\rangle
\Big|^2+\big|\langle \tilde{f},\tilde{z}_{11}\rangle
\big|\Big)\,\E\big(\det(M_0)\big).
\end{equation}
We now turn to $\E (\det (M_0))$.  \\

\noindent
{\bf{Notation} }  For a matrix $M$  and
a subset $S$ (respectively $R$)  of its columns (resp. of its rows),
we denote by $M^S$ (resp. $M_R $) the sub-matrix of $M$ consisting
of the columns in $S$ (resp. the rows in $R$). Also, $M^S_R$ denotes
the matrix that consists in erasing the columns not in
$S$ and the rows not in $R$.\\

\begin{lemma}\label{polcaraccasi}
Let $C=(c_{ij})_{i,j}\in \R^{m\times m}$. For  $q\in \Z$, $1\leq
q\leq m$, and  $\lambda \in \R$  define
$$
C_q(\lambda):= C+\Lambda_q \quad \mbox{where} \quad  \Lambda_q:=\stackrel{\bajo \hspace*{6pt} q
            \hspace*{16pt} m-q}
        {\left(\begin{array}{c|c}
        \lambda \,\Id &0  \\ [3pt]
        \hline
        \alto 0&\Id
      \end{array}\right)}
  \begin{array}{c}
  {\scriptstyle q} \\ [4pt]
  \alto {\scriptstyle m-q}
  \end{array},$$
i.e the matrix obtained by adding $\lambda$ to the first $q$
diagonal entries of $C$.\\
Then,
$$
\det (C_q(\lambda))=\det (C)+\sum _{\ell =1}^q \Big( \sum
_{S\subset\{1,\dots,q\}:\# (S)=\ell } \det \big( C^{\overline
S}_{\overline S}\big)\Big) \lambda^ {\ell}.
$$
where $\overline S$ is the complement set of $S$,  with the
convention that $\det(C^{\emptyset}_{\emptyset})=1. \qquad \diamondsuit$
\end{lemma}

\noindent We set $\lambda: =\sigma \,\sqrt{n}\,\bar{\lambda}$  and write
$M_0= C+  \Lambda $ where
    $$C:=\stackrel{\bajo \hspace*{10pt} n-1
            \hspace*{40pt} n-1}
        {\left(\begin{array}{c|c}
        VV^t+2BB^t &\frac{1}{\sqrt{2}}V \hat{B}^t  \\ [3pt]
        \hline
        \alto \frac{1}{\sqrt{2}}\hat{B}V^t&\frac{1}{2}\hat{B}\hat{B}^t
      \end{array}\right)}
  \begin{array}{c}
  {\scriptstyle n-1} \\ [4pt]
  \alto {\scriptstyle n-1}
  \end{array} \
 \mbox{ and } \ \Lambda:=\stackrel{\bajo \hspace*{6pt} n-1
            \hspace*{16pt} n-1}
        {\left(\begin{array}{c|c}
        \lambda \,\Id &0  \\ [3pt]
        \hline
        \alto 0&\Id
      \end{array}\right)}
  \begin{array}{c}
  {\scriptstyle n-1} \\ [4pt]
  \alto {\scriptstyle n-1}
  \end{array}.
$$
Then, by Lemma~\ref{polcaraccasi} and using that the random variables involved in
the expectation of $M_0$ are the elements of $V$ and
$\bar{\lambda}$, which are independent,  we obtain
\begin{equation} \label{esperv}
\E \big(\det (M_0)\big)
 \ =\ \E \big(\det (C)\big)+\sum _{\ell =1}^{n-1}  \sum
_{\scriptsize{\begin{array}{c}S\subset \{1,\dots,n-1\}\\ \# (S)=\ell
\end{array}}} \E \big( \det \big( C^{\overline S}_{\overline S}\big)
\big) (\sigma \sqrt{n})^{\ell}~\E \big(\bar{\lambda}^ {\ell} \big).
\end{equation}
We now bound the expectations appearing here. We first consider $\E \big(\det (C)\big)$.

\begin{lemma}\label{detbasico} Set $n,k\in \N$, $1\le k<n$. Let  $A=(a_{ij})_{i,j},B\in \R^{k\times
n}$ and $C\in \R^{(n-1)\times n}$. Define
$$
Q:=\stackrel{\bajo \hspace*{20pt} k\hspace*{40pt} n-1}
        {\left(\begin{array}{c|c}
       A\,A^t+B\,B^t&A\,C^t \\ [3pt]
       \hline
      \alto C\,A^t&C\,C^t
       \end{array}\right)}
       \begin{array}{l}
       {\scriptstyle k} \\ [4pt]
       {\scriptstyle n-1}
       \end{array}\ \in \R^{(k+n-1)\times (k+n-1)}.
$$
Then,
\begin{equation*}
\det(Q)\ =\ \det(CC^t)\det(BB^t) \ +\ \sum_{\#(S)=k-1}\Big(
\sum_{i=1}^k \sum _{j=1}^n  (-1)^{i+j-1}  a_{ij} \det (B^{S}_{\overline i})
\det (C^{\overline j})\Big)^2.\qquad \diamondsuit
\end{equation*}
\end{lemma}

\medskip
\noindent  Applying this result
for  $k:=n-1$, $A:=V$, $B:=\sqrt{2}B$ and
$C:=(1/\sqrt{2})\hat{B}$ we get

$$
\det(C)=\det(BB^t)\det(\hat{B}\hat{B}^t)+\sum_{\# (S)=n-2} \Big(
\sum_{i=1}^{n-1}\sum_{j=1}^n (-1)^{i+j-1} v_{ij}\det \big(
\sqrt{2}B^{S}_{\overline i}\big) \det \big(
\frac{1}{\sqrt{2}}\hat{B}^{\overline j}\big) \Big)^2.
$$
Since the random variables $v_{ij}=\sqrt{2/d_j}\,\partial_{1(i+1)}f_j$ are centered and independent,
and since $\Var(v_{ij})=2(d_j-1)$, we obtain
\begin{equation}\label{termino1}
\aligned \E \big(\det(C)
\big)&=\det(BB^t)\det(\hat{B}\hat{B}^t)+\sum_{\# (S)=n-2}\E \left(
\Big( \sum_{i=1}^{n-1}\sum_{j=1}^n\pm v_{ij}\det \big(
\sqrt{2}B^{S}_{\overline i}\big) \det \big(
\frac{1}{\sqrt{2}}\hat{B}^{\overline j}\big) \Big)^2 \right)\\
&=\det(BB^t)\det(\hat{B}\hat{B}^t)+\sum_{\# (S)=n-2} \sum_{i=1}^{n-1}\sum_{j=1}^n 2(d_j-1)
2^{n-2}\Big(\det \big( B^{S}_{\overline i}\big) \Big)^2
\frac{1}{2^{n-1}} \Big( \det \big(\hat{B}^{\overline j}\big) \Big)^2\\
&\leq \det(BB^t)\det(\hat{B}\hat{B}^t)+
(\bD-1)\sum_{\# (S)=n-2} \sum_{i=1}^{n-1}\sum_{j=1}^n \Big(\det \big( B^{S}_{\overline i}\big) \Big)^2
 \Big( \det \big(\hat{B}^{\overline j}\big) \Big)^2\\
&= \det(\hat{B}\hat{B}^t)\Big( \det(BB^t) + (\bD-1) \sum_{i=1}^{n-1}
\det \big( B_{\overline i}B_{\overline i}^t\big)\Big)
\endaligned
\end{equation}
where in the last equality we  applied  twice the well-known Cauchy-Binet
formula, see for example~\cite{Gant}: For $m\le n$,  $A\in \R^{m\times n}$  and $B\in \R^{n\times m}$,
\begin{equation}\label{CB}
  \det (A\,B)= \sum_{S:\# (S)=m}\det (A^S)\det (B_S).
\end{equation}

\noindent Now we compute $\E \big( \det \big( C^{\overline
S}_{\overline
S}\big) \big)$   for $\# (S) =\ell$, $1\le \ell\le n-1$.\\

\noindent $\bullet$ For $\ell =n-1$ it is obvious that
\begin{equation}\label{eq:ell=n-1}
\det \big(C^{\overline S}_{\overline S}\big)
=(1/2^{n-1})\, \det(\hat{B}\hat{B}^t).
\end{equation}

\noindent $\bullet$ For $1 \leq \ell \leq n-2$, we note that
 for each $S\subset \{1,\dots, n-1\}$ with  $\# (S)=\ell$,  we have
$$
 C^{\overline S}_{\overline S}:=
   \stackrel{\bajo \hspace*{20pt} n-1-\ell
            \hspace*{50pt} n-1}
        {\left(\begin{array}{c|c}
       V_{\overline S}(V_{\overline S})^t+2B_{\overline S}(B_{\overline S})^t
      &\frac{1}{\sqrt{2}}V_{\overline S} \hat{B}^t  \\ [4pt]
     \hline
     \alto \frac{1}{\sqrt{2}}\hat{B}(V_{\overline S})^t
  &\frac{1}{2}\hat{B}\hat{B}^t
      \end{array}\right)}
  \begin{array}{l}
  {\scriptstyle n-1-\ell} \\ [4pt]
  \alto {\scriptstyle n-1}
  \end{array}
$$ and we obtain, imitating the computation for the
case $\det(C)$,
\begin{equation}\label{termino2}\E \big(\det(C^{\overline S}_{\overline S})
\big)\le \frac{\det(\hat{B}\hat{B}^t)}{2^\ell}\Big(\det(B_{\overline
S}(B_{\overline S})^t)+ (\bD-1) \sum_{\scriptsize
\begin{array}{cc}1\le i \le n-1-\ell\\i\not\in S\end{array}}\det
\big( B_{\overline{S\cup \{i\}}}(B_{\overline{S\cup
\{i\}}})^t\big)\Big).\end{equation}
Finally  we give an upper-bound for $\E \big( \overline{\lambda}^{\ell} \big)$.
\begin{lemma}\label{momlambda}
Let $G=(a_{ij})_{1\le i,j\le n}$  for $n \geq 2$  be a real random symmetric matrix such that the   the random variables $\{
a_{ij},1 \leq i \leq j \leq n\}$ are independent Gaussian centered,
$Var (a_{ij})=1/n$ for $i \neq j$ and $Var (a_{ij})=2/n$ for $i =
j$, and denote by $\bar{\lambda}$ the maximum between 0 and   the largest
eigenvalue of the matrix $G$.
Then, for $1 \leq \ell \leq n$,
\begin{equation*}
\E \big( \overline{\lambda}^{\ell} \big) \leq 2\cdot 4^{\ell}.\qquad \diamondsuit
\end{equation*}
\end{lemma}

\noindent
Plugging Inequalities (\ref{termino1}), (\ref{termino2}),
\eqref{eq:ell=n-1}  and
Lemma~(\ref{momlambda}) into Formula~\eqref{esperv} we obtain

\begin{equation*}
\aligned \E \big(\det (M_0)\big) \leq & \det(\hat{B}\hat{B}^t) \Big(
\det(BB^t) + (\bD-1) \sum_{i=1}^{n-1} \det \big( B_{\overline
i}B_{\overline i}^t\big) + 2^n(\sigma\sqrt{n})^{n-1}
\\ & + \ \sum _{\ell =1}^{n-2}2^{\ell+1}(\sigma \sqrt{n})^{\ell}  \Big(\sum
_{\# (S)=\ell} \Big(\det(B_{\overline S}(B_{\overline S})^t)+ (\bD-1)
\sum_{i=1}^{n-1-\ell} \det \big( B_{\overline{S\cup
\{i\}}}(B_{\overline{S\cup \{i\}}})^t\big)\Big)\Big)\Big)\\
\leq & \det(\hat{B}\hat{B}^t) \Big( \det(BB^t) + (\bD-1)
\sum_{i=1}^{n-1} \det \big( B_{\overline i}B_{\overline i}^t\big) +
 2^n(\sigma\sqrt{n})^{n-1}
\\ & + \ \sum _{\ell =1}^{n-2}2^{\ell+1}(\sigma \sqrt{n})^{\ell}   \Big(\sum
_{\# (S)=\ell} \det(B_{\overline S}(B_{\overline S})^t)+
(\bD-1)(\ell+1)\sum _{\# (T)=\ell+1}  \det \big(
B_{\overline{T}}(B_{\overline{T}})^t\big)\Big)\Big) .
\endaligned
\end{equation*}
This finally implies, by Identity~(\ref{condexp2}) and
Inequalities~(\ref{inM1}) and (\ref{ineq1}) the inequality we will
focuse on  in next step.

\begin{equation}\label{condexp22}
\aligned
   \E \left(\big|\det(\tilde{L}'')\big|
    \right.& \cdot \chi_{\{\tilde{L}''\succ 0 \}} \Big/ \left.
    L=u,\nabla \tilde{L}=0,\zeta=z\right) \
   \le  \frac{1}{2}\Big( \frac{1}{2}\Big|\left\langle
\frac{\tilde{f}}{\|\tilde{f}\|},\tilde{z}_{11}\right\rangle
\Big|^2+\big|\langle \tilde{f},\tilde{z}_{11}\rangle
\big|\Big)\,\det(\hat{B}\hat{B}^t)\cdot\\
&\cdot \Big( \det(BB^t) + (\bD-1) \sum_{i=1}^{n-1} \det \big(
B_{\overline i}B_{\overline i}^t\big) + 2^{n}(\sigma \sqrt{n})^{n-1}
\\ & + \ \sum _{\ell =1}^{n-2}2^{\ell+1}(\sigma \sqrt{n})^{\ell}   \Big(\sum
_{\# (S)=\ell} \det(B_{\overline S}(B_{\overline S})^t)+
(\bD-1)(\ell+1)\sum _{\# (T)=\ell+1}  \det \big(
B_{\overline{T}}(B_{\overline{T}})^t\big)\Big)\Big) .
   \endaligned
\end{equation}

\bigskip \noindent
\textbf{Step 6}. We  put together the calculations of Steps 4 and 5
to compute an upper bound for $p_{\underline{L}}(u)$ following
Inequality~(\ref{densbound3}). We will also use the following
auxiliary result:

\begin{lemma}\label{bb}  \begin{equation*}\aligned
&\frac{2^{2n-1}}{\mathcal{D}}\det(BB^t) \leq \det(\hat{B}\hat{B}^t)
\leq \frac{2^{2(n-1)}\bD}{\mathcal{D}}    \det(BB^t),
\\
&\frac{2^{2n-1-\ell}}{\mathcal{D}}\det(B_{\overline S}(B_{\overline
S})^t)  \leq \det(\hat{B_{\overline S}}(\hat{B}_{\overline S})^t)
\leq \frac{2^{2(n-1-\ell)}\bD^{\ell+1}}{\mathcal{D}} \det(B_{\overline
S}(B_{\overline S})^t)\quad \mbox{for} \ \ S\subset \{1,\dots,n\},
\#(S)=\ell.
\endaligned
\end{equation*}

\end{lemma}
\proof We have $
  \hat{B}=B~H$ for the diagonal matrix
$$
 H:=\stackrel{\longleftarrow\quad \bajo n\quad \longrightarrow}
        {\left(\begin{array}{ccc}
        \frac{2}{\sqrt{d_1}}&  &   \\
       & \ddots&  \\
      & & \frac{2}{\sqrt{d_n}}
      \end{array}\right)}
   \begin{array}{c}\uparrow\\ {\scriptstyle{n}}\\\downarrow \end{array}.
$$
By  Cauchy-Binet formula \eqref{CB},
\begin{eqnarray*}
  \det(\hat{B}\hat{B}^t)&=&\sum_{k=1}^n
  \det(\hat{B}^{\overline k})\det((\hat{B}^{\overline k})^t)\\
& = & \ \sum_{k=1}^n \Big(\det(B^{\overline k}H_{\overline
k}^{\overline k})\big)^2\
 = \ \sum_{k=1}^n \big(\det(H_{\overline k}^{\overline k})\big)^2 \big(\det({B}^{\overline k})\big)^2
 \\
 &=& \sum_{k=1}^n \frac{2^{2(n-1)}d_k}{\mathcal{D}} \big(\det({B}^{\overline
 k})\big)^2.
\end{eqnarray*}
 The proof concludes using
 $$\frac{2^{2n-1}}{\mathcal{D}} \le \frac{2^{2(n-1)}d_k}{\mathcal{D}} \le \frac{2^{2(n-1)} \bD}{\mathcal{D}}
 \
 \ \mbox{since} \ \ d_k\ge 2
 \quad\mbox{and} \quad \sum_{k=1}^n
  \big(\det({B}^{\overline k})\big)^2=\det (BB^t).$$
The proof of the second assertion is analogous. \eproof
\\

\noindent According to Inequalities~(\ref{densbound3}),
(\ref{condexp22}), and Identity~(\ref{factdensidad2}), we get:
\begin{equation*}
  \aligned  p_{\underline{L}}(u)  & \le \sigma_V(V)\, \int_{(S\times \R )^n}
   \E \left(\big|\det(\tilde{L}'')\big|\cdot
   \chi_{\{\tilde{L}''\succ 0 \}} \Big/
   L=u,\nabla \tilde{L}=0,\zeta=z \right) \cdot
  p_{L,\nabla \tilde{L},\zeta}(u,0,z)~dz\\
  & \le \sigma_V(V)\, \int_{(S\times \R )^n}
  \frac{1}{2}\Big( \frac{1}{2}\Big|\left\langle
\frac{\tilde{f}}{\|\tilde{f}\|},\tilde{z}_{11}\right\rangle
\Big|^2+\big|\langle \tilde{f},\tilde{z}_{11}\rangle
\big|\Big)\,\det(\hat{B}\hat{B}^t)\cdot\\
&\cdot \Big( \det(BB^t) + (\bD-1) \sum_{i=1}^{n-1} \det \big(
B_{\overline i}B_{\overline i}^t\big) + 2^{n}(\sigma \sqrt{n})^{n-1}
\\ & + \ \sum _{\ell =1}^{n-2}2^{\ell+1}(\sigma \sqrt{n})^{\ell}   \Big(\sum
_{\# (S)=\ell} \det(B_{\overline S}(B_{\overline S})^t)+
(\bD-1)(\ell+1)\sum _{\# (T)=\ell+1}  \det \big(
B_{\overline{T}}(B_{\overline{T}})^t\big)\Big)\Big) \cdot\\
&  \frac{e^{-u/2}}{(2\pi)^{n}2^{n-1}|\det (A)|\big(\det
    (B B^t)\big)^{1/2}\,\sqrt{u}}\,\cdot\,p_{\zeta}(z)dz.
  \endaligned
\end{equation*}
Here we notice that $|\det(A)|$ is the $n$-volume of the
parallelotope generated in $\R^n$ by the rows of $A$, that is, in
the same way we computed  $\det(M_2)$ in (\ref{detm2}), we have
$$
 |\det(A)|=\dist \left(\frac{1}{\sqrt{2}}\tilde{z}_{11},\tilde{S}\right)\,
          \det(\hat{B} \hat{B}^t ) ^{1/2}
          = \frac{1}{\sqrt{2}}\,\Big|\left\langle
  \frac{\tilde{f}}{\|\tilde{f}\|},\tilde{z}_{11}\right\rangle \Big|  \big(\det(\hat{B} \hat{B}^t ) \big)^{1/2}
$$
where like previously $\tilde S:=\mbox{span}(\tilde z_2,\dots,\tilde
z_n)\subset \R^{n}$ is the hyperplane spanned by the the rows of
$\hat B$. Therefore, using Cauchy-Schwartz inequality for $\langle
\tilde f/\|\tilde f\|,\tilde z_{11}\rangle$, applying Lemma~\ref{bb}
and the fact that $2^n\le \mathcal{D}$, we get

\begin{equation*}
\aligned p_{\underline{L}}(u)&\leq \sigma _V (V)\int_{(S\times
\R)^n} \frac{\sqrt
2}{(4\pi)^n}\big(\frac{1}{2}\|\tilde{z}_{11}\|+\|\tilde{f}\|
\big)\Big( \frac{\det(\hat{B}\hat{B}^t)}{\det(BB^t)}
\Big)^{1/2}\cdot\\
&\cdot \Big( \det(BB^t) + (\bD-1) \sum_{i=1}^{n-1} \det \big(
B_{\overline i}B_{\overline i}^t\big) + 2^{n}(\sigma \sqrt{n})^{n-1}
\\ & + \ \sum _{\ell =1}^{n-2}2^{\ell+1}(\sigma \sqrt{n})^{\ell}   \Big(\sum
_{\# (S)=\ell} \det(B_{\overline S}(B_{\overline S})^t)+
(\bD-1)(\ell+1)\sum _{\# (T)=\ell+1}  \det \big(
B_{\overline{T}}(B_{\overline{T}})^t\big)\Big)\Big) \cdot \frac{e^{-u/2}}{\sqrt u} p_{\zeta}(z)dz\\
&\leq \sigma _V (V)\int_{(S\times \R)^n} \frac{\sqrt
2}{(4\pi)^n}\big(\frac{1}{2}\|\tilde{z}_{11}\|+\|\tilde{f}\|
\big)2^{n-1}\frac{\sqrt \bD}{\sqrt \mathcal{D}}\cdot\\
&\cdot \Big( \frac{\mathcal{D}}{2^{2n-1}}\det(\hat{B}\hat{B}^t)   +
(\bD-1) \sum_{i=1}^{n-1}\frac{\mathcal{D}}{2^{2n-2}}
\det(\hat{B_{\overline i}}(\hat{B}_{\overline i})^t)
 + \mathcal{D}(\sigma
\sqrt{n})^{n-1}
\\ & + \ \sum _{\ell =1}^{n-2}2^{\ell+1}(\sigma \sqrt{n})^{\ell}  \Big(\sum
_{\# (S)=\ell} \frac{\mathcal{D}}{2^{2n-1-\ell}}\det(\hat
B_{\overline S}(\hat B_{\overline S})^t)+ (\bD-1)(\ell+1)\sum _{\#
(T)=\ell+1}  \frac{\mathcal{D}}{2^{2n-\ell}}\det \big(
\hat B_{\overline{T}}(\hat B_{\overline{T}})^t\big)\Big)\Big) \cdot\\ & \cdot \frac{e^{-u/2}}{\sqrt u} p_{\zeta}(z)dz\\
&\leq \sigma _V (V)\int_{(S\times \R)^n} \frac{\sqrt
2}{(8\pi)^n}\big(\frac{1}{2}\|\tilde{z}_{11}\|+\|\tilde{f}\|
\big)\sqrt {\bD\mathcal{D}}\cdot \Big( \det(\hat{B}\hat{B}^t) + 2(\bD-1)
\sum_{i=1}^{n-1}\det(\hat{B_{\overline i}}(\hat{B}_{\overline i})^t)
 + 2(4\sigma
\sqrt{n})^{n-1}
\\ & + \ \sum _{\ell =1}^{n-2}(4\sigma \sqrt{n})^{\ell}  \Big(\sum
_{\# (S)=\ell}2 \det(\hat B_{\overline S}(\hat B_{\overline S})^t)+
(\bD-1)(\ell+1)\sum _{\# (T)=\ell+1}\det \big( \hat
B_{\overline{T}}(\hat B_{\overline{T}})^t\big)\Big)\Big)
\cdot\frac{e^{-u/2}}{\sqrt u}
p_{\zeta}(z)dz\\[2mm]
&= \E\big(H(u,\zeta)\big),
\endaligned
\end{equation*}
where

\begin{equation*}
\aligned H(u,\zeta):&=\frac{\sqrt 2}{(8\pi)^n}\sigma _V (V)
\big(\frac{1}{2}\|\tilde{\zeta}_{11}\|+\|\tilde{f}\| \big)\sqrt
{\bD\mathcal{D}}\cdot \Big( \det(\hat{B}({\zeta})\hat{B}^t({\zeta})) +
 2(4\sigma \sqrt{n})^{n-1}\\
& + 2(\bD-1) \sum_{i=1}^{n-1}\det(\hat B_{\overline
i}({\zeta})(\hat{B}_{\overline i})^t({\zeta}))
 + \ \sum _{\ell =1}^{n-2}(4\sigma \sqrt{n})^{\ell}  \Big(\sum
_{\# (S)=\ell} 2\det(\hat B_{\overline S}({\zeta})(\hat B_{\overline
S})^t({\zeta})) \\ &+ (\bD-1)(\ell+1)\sum _{\# (T)=\ell+1}\det \big(
\hat B_{\overline{T}}({\zeta})(\hat
B_{\overline{T}}({\zeta}))^t\big)\Big)\Big)
\cdot\frac{e^{-u/2}}{\sqrt u}.
\endaligned
\end{equation*}
Here
$$
     \hat B(\zeta):=\stackrel{\longleftarrow\quad \bajo n\quad \longrightarrow}
        {\left(\begin{array}{ccc}
        \frac{2}{\sqrt{d_1}}\partial_2 f_1 &\dots &
        \frac{2}{\sqrt{d_n}}\partial_2 f_n \\
        \vdots& &\vdots \\
        \frac{2}{\sqrt{d_1}}\partial_n f_1 &\dots &
        \frac{2}{\sqrt{d_n}}\partial_n f_n
        \end{array}\right)}\begin{array}{c}\uparrow\\ {\scriptstyle{n-1}}\\\downarrow \end{array}
        \quad  \mbox{and} \quad
 \tilde{\zeta}_{11}:=\big( \frac{2}{\sqrt{d_1}}
\partial_{11}f_1, \dots,  \frac{2}{\sqrt{d_n}}
\partial_{11}f_n\big).
$$
Our next goal is then to bound $\E(H(u,\zeta))$. We first note that
the matrix $\hat{B}(\zeta)$ is  independent from $\tilde{\zeta}_{11}
$, so that the expectation
can be factorized as a product of expectations. \\

\noindent First,  using Lemma \ref{lemcovar} and the definition of
$\tilde{f}$ we easily get
$$
\E \big( \frac{1}{2}\|\tilde{\zeta}_{11}\|+\|\tilde{f}\| \big) \leq
\sqrt{2(\bD-1)n}+\sqrt{\bD u}.
$$
For the other expectations we   apply the following.

\begin{lemma}(e.g.~\cite[Lemma~13.6]{AW2})
Set $m\le n$ and let $U$ be an $m\times n$
random matrix whose elements are independent
real standard normal. Then
\begin{equation}\tag*{\qed}
\E \big( \det (UU^t) \big)=\frac{n!}{(n-m)!}.
\end{equation}
\end{lemma}

\noindent Therefore, since by Lemma~\ref{lemcovar},
$\frac{1}{2}\hat{B}(\zeta)$ satisfies the hypothesis of the lemma
with $m=n-1$, we obtain
$$
\E \big( \det(\hat{B}(\zeta)\hat{B}^t(\zeta) \big)=4^{n-1}n!
$$
and we get similar expressions for the other determinants in
$\E(H(u,\zeta)))$: \begin{equation*}\aligned & \E\big(\det(\hat
B_{\overline i}({\zeta})(\hat{B}_{\overline i})^t({\zeta}))\big)=
4^{n-2}\frac{n!}{2},\\ &\E \big( \det (\hat{B}_{\overline
S}(\zeta)\hat{B}^t_{\overline S}(\zeta))
\big)=4^{n-1-\ell}\frac{n!}{(\ell+1)!},\\&\E\big( \det \big( \hat
B_{\overline{T}}({\zeta})(\hat
B_{\overline{T}}({\zeta}))^t\big)=4^{n-2-\ell}\frac{n!}{(\ell+2)!}.
\endaligned\end{equation*}
We also apply  Formula~(\ref{volumenV}): $\sigma _V(V)=4\sqrt 2
\pi^{n+\frac{1}{2}}/\big(\Gamma(n/2)\Gamma((n+1)/2)\big)$. Therefore

\begin{equation}\label{expH}
\aligned \E \big(& H(u,\zeta) \big) = \frac{\sqrt 2}{(8\pi)^n}
\frac{4\sqrt 2 \pi^{n+\frac{1}{2}}}{\Gamma(n/2)\Gamma((n+1)/2)}
 \big(\sqrt{2(\bD-1)n}+\sqrt{\bD u}\big)\sqrt{\bD\mathcal{D}} \\
&\cdot \Bigg(\,4^{n-1}n! + 2(4\sigma \sqrt n)^{n-1}  +
2(\bD-1)\sum_{i=1}^{n-1}4^{n-2}\frac{n!}{2} \\ & + \sum_{\ell
=1}^{n-2}(4\sigma\sqrt{n})^{\ell}\Big({n-1\choose
\ell}2\cdot4^{n-1-\ell}\frac{n!}{(\ell+1)!}+
(\bD-1)(\ell+1){n-1\choose \ell+1} 4^{n-2-\ell}\frac{n!}{(\ell+2)!}
\Big)\Bigg)\\
&\cdot \frac{e^{-u/2}}{\sqrt{u}}\\
&= \ \frac{\sqrt{\pi}}{8^{n-1}\Gamma (n/2)\Gamma
((n+1)/2)} \big(\sqrt{2(\bD-1)n}+\sqrt{\bD u}\big)\sqrt{\bD\mathcal{D}}\,4^{n-1}n!\\
&\ \cdot \Bigg( 1+ 2\frac{(\sigma \sqrt n)^{n-1}}{n!}
+\frac{(\bD-1)(n-1)}{4} + \sum_{\ell
=1}^{n-2}(\sigma\sqrt{n})^{\ell}\Big({n-1\choose
\ell}\frac{2}{(\ell+1)!}\\
&\ +\frac{(\bD-1)(\ell+1) }{4}{n-1\choose \ell+1}
\frac{1}{(\ell+2)!}
\Big)\Bigg) \cdot \frac{e^{-u/2}}{\sqrt{u}}\\
&\le \ \frac{\sqrt{\pi}}{2^{n-1}\Gamma (n/2)\Gamma
((n+1)/2)} \big(\sqrt{2(\bD-1)n}+\sqrt{\bD u}\big)\sqrt{\bD\mathcal{D}}\,n!\\
&\ \cdot \left(  \sum_{\ell =0}^{n-1}{n-1\choose
\ell}(\sigma\sqrt{n})^{\ell}+\frac{(\bD-1)(n-1)
}{4}\sum_{\ell=0}^{n-2} {n-2\choose \ell}(\sigma \sqrt n)^\ell
\right)\, \cdot \,\frac{e^{-u/2}}{\sqrt{u}} .
\endaligned
\end{equation}

\noindent Now, we assume $n\ge 3$ and we bound this expectation  for
$ 0 < u < 1/(4\bD^2 n^5) $ in which case, by the bound for $\sigma^2$
given in~(\ref{varw}), $\sigma^2\le  4\bD(\bD-1)u \le 1/n^5$.\\ We will
use throughout the bounds $1+x\le e^x$ for any $x$ and $e^x-1\le 2x $
for $0\le x\le 1$.

\smallskip
\noindent The factorial term $n! = \Gamma(n+1)$ and the other Gamma functions
in the first line of
the right-hand side of Inequality~(\ref{expH}) can be bounded
through Stirling's formula
\cite[Formula~6.1.38]{AbSt}: for any $x>0$,
\[
\Gamma(x+1)=
\sqrt{2\pi x} \Big(\frac{x}{e}\Big)^x e^{\theta/(12x)}
 \quad
\mbox{for some $0 < \theta = \theta(x) < 1$.}
\]
so that,
\[
\sqrt{2\pi x} \Big(\frac{x}{e}\Big)^x
<
\Gamma(x+1)
<
\sqrt{2\pi x} \Big(\frac{x}{e}\Big)^x e^{1/(12x)}
.
\]
Also,
$$\sqrt{\bD u}\le \frac{1}{2\sqrt \bD n^{5/2}}\le \sqrt{2(\bD-1)n}\Big(\frac{1}{2\sqrt{2(\bD-1)n}\,\sqrt \bD n^{5/2}} \Big)\le
\frac{\sqrt{2(\bD-1)n} } {4n^3 }, $$ which implies
$$
   \sqrt{2(\bD-1)n}+\sqrt{\bD u}\le  \sqrt{2(\bD-1)n}\,
   \Big(1+ \frac{1}{4n^3}\Big).
$$
Therefore, the first line of the right-hand side
of Inequality~(\ref{expH}) satisfies
\begin{equation*}
\aligned &\frac{\sqrt{\pi}}{2^{n-1}\Gamma
(n/2)\Gamma ((n+1)/2)}\big(\sqrt{2(\bD-1)n}+\sqrt{\bD u}\big)
\sqrt{\bD\mathcal{D}}\,n!\\
& \le \ \frac{\sqrt{\pi}}{2^{n-1}\sqrt{(n-2)\pi}\sqrt{(n-1)\pi}}
    \Big(\frac{2e}{n-2}\Big)^{\frac{n-2}{2}}
    \Big(\frac{2e}{n-1}\Big)^{\frac{n-1}{2}}
    \sqrt{2(\bD-1)n}\,\Big(1+ \frac{1}{4n^3}\Big)\\
&\qquad \sqrt{\bD\mathcal{D}}\sqrt{2\pi n}
\Big(\frac{n}{e}\Big)^n e^{1/(12n)}\\
 & = \ \frac{e^{1/(12n)}}{
 e^{3/2}}\sqrt{\frac{n^2}{(n-2)(n-1)}}\Big(\frac{n}{n-2}\Big)^{\frac{n-2}{2}}
 \Big(\frac{n}{n-1}\Big)^{\frac{n-1}{2}}n^{3/2}
 \sqrt{2(\bD-1)}\,\Big(1+ \frac{1}{4n^3}\Big)\sqrt{\bD\mathcal{D}}
 \\
& \le \ 3\, \bD \,\sqrt{\mathcal{D}}\,n^{3/2}\,
\,\Big(1+\frac{1}{4n^3}\Big)\,\Big(1 + \frac{1}{6n}\Big) \ \le 4 \,\bD
\,\sqrt{\mathcal{D}}\,n^{3/2}.
\endaligned
\end{equation*}
We now turn our attention to the  term under brackets in the right-hand side of
Inequality~(\ref{expH}).

\smallskip \noindent
We have $\sigma\sqrt n \le 1/n^2$. Therefore \begin{equation*}
\aligned &
 \sum_{\ell =0}^{n-1}{n-1\choose \ell}(\sigma\sqrt{n})^{\ell}
  +\frac{(\bD-1)(n-1)}{4}
  \sum_{\ell=0}^{n-2} {n-2\choose \ell}(\sigma \sqrt n)^\ell\\
  &\le \ \Big(1+\frac{1}{n^2}\Big)^{n-1} + \frac{(\bD-1)(n-1) }{4}
  \Big(1+\frac{1}{n^2}\Big)^{n-2}\\
& \le \ e^{\frac{n-2}{n^2}}\Big( 1+\frac{1}{n^2} +  \frac{(\bD-1)(n-1)
}{4}\Big) \ \le \ \Big(1+
\frac{2(n-2)}{n^2}\Big)\,\Big(1+\frac{1}{n^2} + \frac{(\bD-1)(n-1)
}{4}\Big) \ \le n\,\bD.
\endaligned\end{equation*}
Adding up, since   $e^{-u/2}\le 1$,  we obtain
\begin{equation*}\label{cotafinp}
p_{\underline{L}}(u) \ \leq \ \E\big(H(u,\zeta)\big) \ \le \
4\,\bD^2\, {\mathcal{D}}^{1/2}\,n^{5/2} \,  \frac{1}{\sqrt{u}}.
\end{equation*}

\bigskip
\noindent\textbf{Step 7}. We finally complete the proof of
Theorem~\ref{boundkapatilde}.\\

\noindent For $0<\alpha <1/(4\bD^2n^5)$, the previous estimate for
$p_{\underline{L}}(u)$ implies

$$
\P(\underline{L}<\alpha)=\int_0^{\alpha}p_{\underline{L}}(u)\, du\
\leq 8\,\bD^2\, {\mathcal{D}}^{1/2}\,n^{5/2} \,\sqrt\alpha.
$$
Let us go back to the starting inequality~(\ref{des1}):
$$
   \P \left(\tilde\kappa (f)>a \right)\leq \P
   \left(\underline{L}<\frac{1}{a^2} (1+\ln a )N \right)+
   \exp\left(-\frac{N}{2}(\ln a - \ln (\ln a +1)\right).
$$
 where we recall that $$N  = \sum_{i=1}^n{n+d_i\choose n}\ \le \
n^{\bD+2}.$$ By hypothesis in the theorem, $a>a_n:=4\,\bD^2n^3N^{1/2}$. \\
We set $\alpha:=(1+\ln a)N/a^2 $ and verify $\alpha <1/(4\bD^2n^5)$.
It is enough to verify it with $a_n$:
$$\frac{(1+\ln a_n)N}{a_n^2} < \frac{1}{4\bD^2n^5} \iff
1+\ln a_n \le
4\bD^2n$$ which is satisfied since for $\bD\ge 2$ and $n\ge 3$,
$$1+\ln a_n<1+\ln(4\bD^2)+3\ln n + \frac{\bD+2}{2}\ln n \le 4\bD^2 +
\big(\frac{\bD}{2}+4\big)\ln n< 4\bD^2+4\bD^2(n-1)=4\bD^2n.
$$
Therefore, by Inequality~(\ref{des1}),

\begin{equation*}
\aligned \P (\tilde{\kappa} (f))>a )&\leq \ \P \Big( \underline{L}<
\frac{(1+\ln a)N}{a^2}\Big)+\exp \Big( -\frac{N}{2} \big(\ln a - \ln (\ln a +1)\big) \Big)\\
&\leq \ 8\,\bD^2\, {\mathcal{D}}^{1/2}\,n^{5/2}
\,\sqrt\alpha+\frac{1}{a}\ =\ K_n\frac{(1+\ln a)^{1/2}}{a}
\endaligned
\end{equation*}
where $ K_n=8\bD^2{\mathcal{D}}^{1/2}\,{N}^{1/2}n^{5/2}+1.$ Here we
used $ \exp \big( (-N/2) \big(\ln a - \ln (\ln a
+1)\big)\big)<1/a$ for $a>2, N>10$.
This proves part (i) of Theorem \ref{boundkapatilde}.\\

\noindent (ii) We verify that  $K_n> a_n$. It is enough to check
$$8\,\bD^2{\mathcal{D}}^{1/2}{N}^{1/2}n^{5/2}\ge 4\,\bD^2n^3N^{1/2}
\iff 2 \,{\mathcal{D}}^{1/2}\ge  n^{1/2} \iff 4\,{\mathcal{D}}\ge
n$$ which holds because $4\,{\mathcal{D}}\ge 4\cdot 2^n\ge  n$.\\
Therefore we can write
\begin{equation*}
\aligned \E (\ln \tilde{\kappa} (f))&= \ \int_0^{+\infty}\P (\ln
\tilde{\kappa} (f)>x )\,dx \ \leq \
\ln K_n +\int_{\ln K_n}^{+\infty}\P \big(\tilde{\kappa} (f)>e^x \big)\,dx\\
&\leq \ \ln K_n +\int_{\ln K_n}^{+\infty}K_n(1+x)^{1/2}e^{-x}\, dx\\
&\leq \ \ln K_n +K_n\int_{\ln K_n}^{+\infty}x^{1/2}e^{-x}\, dx+\frac{K_n}{2}\int_{\ln K_n}^{+\infty}x^{-1/2}e^{-x}\,dx\\
&= \ \ln K_n + K_n(e^{-\ln K_n}\,(\ln K_n)^{1/2})+ K_n\,\int_{\ln K_n}^{+\infty}x^{-1/2}e^{-x}\,dx\\
&\le \  \ln K_n +(\ln K_n)^{1/2}+K_n(\ln K_n)^{-1/2}\,\int_{\ln K_n}^{+\infty}e^{-x}\,dx\\
&= \ \ln K_n +(\ln K_n)^{1/2}+ (\ln K_n)^{-1/2}.
\endaligned
\end{equation*}
Here we used the inequality
$(1+x)^{1/2}<x^{1/2}+\frac{1}{2}x^{-1/2}$  for $x>0$ and integration
by parts.

\bigskip

\section{Auxiliary lemmas}\label{auxlemmas}
This section contains the proofs of all the auxiliary results indicated by the symbol $\diamondsuit$, which were  stated without proof during the text.

\medskip
\proofof{Lemma~\ref{largedev}}  According to the definition of the
Weyl norm,
\begin{equation}\label{weylchi2}
   \|f\|_W^2=\sum_{i=1}^n\sum _{ |j |=d_i}\xi_{i,j}^2
\end{equation}
where, due to the distribution, the random variables
$$
   \xi_{i,j}=\frac{a_j^{(i)}}{{d_i \choose j }^{1/2}}
$$
are independent identically distributed (i.i.d.) standard normal.\\
It is easy to see that the number of terms in the sum
(\ref{weylchi2}) is equal to $N $, so that
$$
   \P \left(\|f\|_W^2 \geq (1+\eta)N \right)
  =\P\left((\xi_1^2-1)+\cdots
   +(\xi_{N}^2-1)\ge \eta N \right)
  =\P \left(\frac{X_1+\ldots+X_{N}}{N}\geq \eta \right)
$$
where $X_{1},\ldots ,X_{N}$ are i.i.d. random variables having the
distribution of $\xi ^{2}-1$, $\xi $ a normal standard random
variable.
\\
The logarithmic moment generating function of $\xi ^{2}-1$ is
\begin{equation*}
   \Lambda (\lambda )=\ln \E \{e^{\lambda (\xi ^{2}-1)}\}
  =\left\{
   \begin{array}{ll}
     -\lambda -\frac{1}{2}\ln (1-2\lambda )
         & \mbox{if $\lambda <\frac{1}{2}$} \\[4pt]
    +\infty & \mbox{if $\lambda \geq \frac{1}{2}$}
\end{array}
\right.
\end{equation*}
and its Fenchel-Legendre transform
\begin{equation*}
   \Lambda ^{\ast }(x)
 =\sup_{\lambda \in \R}(\lambda x-\Lambda (\lambda))
 =\left\{
      \begin{array}{ll}
       \frac{1}{2}(x-\ln (x+1)) & \mbox{if $x>-1$} \\
       +\infty & \mbox{if $x\leq -1$.}
     \end{array}\right.
\end{equation*}
A basic result on large deviations \cite[Ch.~2]{DemboZeitouni}
states that, for any integer $m$ and any $x>0$,
\begin{equation*}
  \P\left( \frac{X_{1}+\cdots +X_{m}}{m}\geq x\right) \leq
  \exp({-m\Lambda ^{\ast }(x)}).
\end{equation*}
This implies the statement.
\eproof
\bigskip

\noindent

\proofof{Lemma~\ref{lemcovar}}
For the first item, from the fact that
$\E(a_ja_{j'})=\E(a_j)\E(a_{j'})=0$ for $j\ne j'$
(by  the independence of the $a_j$), we have
$$
    \E (f(x)f(y))=\E \left(\sum_{j,j'} a_j a_{j'} x^j y^{j'}\right)
    =\sum_j\E((a_j)^2) x^jy^j
    =\sum_j{d\choose j} x^jy^j=\langle x,y\rangle^{d}.
$$
For the following items, we observe that we can differentiate
under the expectation sign
the function $(x,y)\mapsto \E(f(x)f(y))=\langle x,y\rangle^d$, e.g.
$$
   \aligned
  &\E\left(f(x)\partial_k f(y)\right)
  =\frac{\partial (\langle x,y\rangle^d)}{\partial y_k}(x,y)
  = dx_k\langle x,y\rangle^{d-1}\\
  &\E\left(\partial_k f(x)
     \partial_{k'}f(y)\right)
   =\partial_{kk'}^2 (\langle x,y\rangle^d)
  = \delta_{kk'}d\langle x,y\rangle^{d-1}+d(d-1)x_{k'}y_k
       \langle x,y\rangle^{d-2}.
\endaligned
$$
This gives the covariances when specializing $x=y=e_0$.
\eproof

%
\bigskip

\noindent Our next lemma deals with the analytic description of the
geometry of the manifold $V$ which is used in the proof of Lemma~\ref{lemmaM}. We define the function
$\psi:B_{2n-1,\delta}  \rightarrow \R^{n+1}\times\R^{n+1}$ by means of:
$$
\psi(\sigma_2,\ldots,\sigma _n,\tau _2,\ldots,\tau _n, \theta)=\left(\frac{C}{\|C\|_{n+1}},
   \frac{D}{\|D\|_{n+1}}\right),
$$
where $B_{2n-1,\delta} $ is the open ball in $\R^{2n-1}$, centered at
the origin and radius
$\delta$ sufficiently small, $\|.\|_{n+1}$ is the Euclidean norm in
$\R^{n+1}$ and the definition of $C$ and $D$ is given in
several steps by the following:
\begin{itemize}
\item
We set $\sigma_1:=(1-\sigma_2^2-\ldots-\sigma _n^2)^{1/2},
~\tau_1:=(1-\tau_2^2-\ldots-\tau _n^2)^{1/2}$,\\[2mm]
$a(\sigma,\tau):=-\left( \sum_{j=2}^n\sigma _j \tau _j \right)/(\sigma
_1 + \tau _1)$, $~~n(\sigma, \tau ):=\sqrt{1+a^2(\sigma ,\tau)}$.
\item
$A:=\frac{1}{n(\sigma ,\tau )}\left( \sigma _1e_0+\sum_{j=2}^n
  \sigma _je_j+a(\sigma ,\tau)e_1\right)$, and

$B:=\frac{1}{n(\sigma ,\tau )}\left( \tau_1e_1+\sum_{j=2}^n
  \tau _je_j+a(\sigma ,\tau) e_0\right)$.

\item
$C:=\cos(\theta/\sqrt{2})A+ \sin(\theta/\sqrt{2})\sigma _1e_1$,
and $D:=\cos(\theta/\sqrt{2})B- \sin(\theta/\sqrt{2})\tau _1e_0$.
\end{itemize}

\begin{lemma}\label{curvaturas}\textbf{[Geometry of $V$]}
\begin{enumerate}
\item
$\psi$ is a parametrization of a neighborhood of the point
$(e_0,e_1)$ in the manifold $V$ with $\psi (0)=(e_0,e_1)$.
\item
For $2\le j\le n$,
$$
\frac{\partial \psi}{\partial \sigma_{j}}(0)=(e_{j},0),~~~
  \frac{\partial \psi}{\partial \tau_{j}}(0)=(0,e_{j})\
  \mbox{ and } \  \frac{\partial \psi}{\partial\theta }(0)
 =\frac{1}{\sqrt{2}}(e_1,-e_0).
$$

Therefore  the orthonormal basis $\mathcal{B}_T$ (defined in
\eqref{basee0e1}) of the tangent space of $V $ at the point
$(e_0,e_1)$ satisfies
$$
  \mathcal{B}_T=\left(\frac{\partial\psi}{\partial \sigma _{2}}(0),
   \dots,\frac{\partial \psi}{\partial\sigma _{n}}(0),\frac{\partial \psi}
   {\partial \tau_{2}}(0),\dots,\frac{\partial \psi}{\partial \tau _{n}}(0),
   \frac{\partial \psi}{\partial \theta }(0)\right).
$$
\item The curvatures are given by:
\begin{equation*}\aligned
  &\frac{\partial^2\psi}{\partial\sigma _{j}^2}(0)=(-e_0,0);\
  \frac{\partial^2\psi}{\partial\tau _{j}^2}=(0,-e_1);\
  \frac{\partial^2\psi}{\partial\sigma_{j}\partial\tau _{j}}
 =-\frac{1}{2}(e_1,e_0) \quad \mbox{for} \  2\le j\le n,\\
 &\frac{\partial^2\psi}{\partial\sigma _{j}\partial\sigma_{k}}(0)
 =~\frac{\partial^2\psi}{\partial\tau _{j}\partial\tau_{k}}(0)
 =~\frac{\partial^2\psi}{\partial\sigma _{j}\partial\tau_{k}}(0)
 =(0,0)\quad \mbox{for} \ 2\le j\ne k\le n,\\
 &\frac{\partial^2\psi}{\partial \theta^2}(0)
 =-\frac{1}{2}(e_0,e_1);\ \frac{\partial^2\psi}
   {\partial \sigma_{j}\partial \theta}(0)
 =\frac{\partial^2\psi}{\partial \tau_{j}\partial \theta}
 =(0,0) \quad \mbox{for} \  2\le j\le n.
\endaligned
\end{equation*}
\end{enumerate}
\end{lemma}

\proof If $\delta$ is small enough, $\psi$ is well defined
and is $\scC^{\infty}$. It is easy to check that
$\langle C,D\rangle _{\R^{n+1}}=0$, so that
$\psi (\sigma_2,\ldots,\sigma_n,\tau_2,\ldots,\tau _n, \theta)\in V$.
\\
A routine calculation of first derivatives allows to check {\em 2}
and also implies that if $\delta$ is small enough, $\psi $ is a
diffeomorphism from $B(0,\delta )$ onto its image. The computation
of second order derivatives is also immediate. \eproof

\begin{corollary}\label{matrizM}
Let us  set  $L':=L'(e_0,e_1)$ and $L'':=L'' (e_0,e_1)$
for the  free first order and second order
derivatives of $L$ at $(e_0,e_1)$. We use the parametrization introduced
in the previous Lemma. Consider the function
$$
\tilde{L}(\sigma_2,\ldots,\sigma _n,\tau _2,\ldots,\tau _n, \theta)=L\big(\psi(\sigma_2,\ldots,\sigma _n,\tau _2,\ldots,\tau _n, \theta) \big)
$$
Let $M$ be the symmetric matrix of the linear operator $\tilde{L}''(0)$ in the canonical basis of $\R^{2n-1}$:
$$
   M=\left(
    \begin{array}{ccc}
      M_{\sigma \sigma} & M_{\sigma \tau} & M_{\sigma \theta } \\
      M_{\tau \sigma} & M_{\tau \tau} & M_{\tau \theta} \\
      M_{\theta \sigma} & M_{\theta \tau} & M_{\theta \theta} \\
    \end{array}\right) \in \R^{(2n-1)\times(2n-1)}
$$
where for $2\le j,k\le n$,
\begin{eqnarray*}
  (M_{\sigma\sigma})_{jk} = (M_{\sigma\sigma})_{kj}
  &=&
  \frac{\partial^2(L\circ \psi )}{\partial \sigma _{j}
  \partial\sigma_{k}}
  = \left\langle L''\frac{\partial \psi}{\partial \sigma_{j}}(0),
    \frac{\partial \psi}{\partial \sigma _{k}}(0)\right\rangle +
    \left\langle L', \frac{\partial^2\psi}{\partial\sigma_{j}
      \partial\sigma _{k}}(0)\right\rangle\\
 &=& \left\{\begin{array}{lcl}\langle L''(e_j,0),(e_j,0)\rangle -
   \langle L', (e_0,0)\rangle &\mbox{for} & j=k\\
   \langle L''(e_j,0),(e_k,0)\rangle &\mbox{for} & j\ne k
     \end{array}\right. \\
 &=& \left\{\begin{array}{lcl}\frac{\partial^2L}
   {\partial x_{j}^2}-\frac{\partial L}{\partial x_0}
   &\mbox{for} & j=k\\[2mm]
   \frac{\partial^2L}{\partial x_{j}\partial x_{k}}
   &\mbox{for} & j\ne k\end{array}\right.
\end{eqnarray*}

\begin{eqnarray*}
  (M_{\sigma\tau})_{jk} = (M_{\tau\sigma})_{kj}
  &=&\frac{\partial^2(L\circ \psi )}
     {\partial \sigma _{j}\partial\tau_{k}}
  = \left\langle L''\frac{\partial \psi}{\partial \sigma_{j}}(0),
  \frac{\partial \psi}{\partial \tau_{k}}(0)\right\rangle
  + \left\langle L', \frac{\partial^2\psi}{\partial\sigma_{j}
       \partial\tau _{k}}(0)\right\rangle\\
  &=& \left\{\begin{array}{lcl}\langle L''(e_j,0),(0,e_j)\rangle
     -\frac{1}{2} \langle L', (e_1,e_0)\rangle
     &\mbox{for} & j=k\\
     \langle L''(e_j,0),(0,e_k)\rangle
     &\mbox{for} & j\ne k\end{array}\right.
\\
  &=& \left\{\begin{array}{lcl}\frac{\partial^2L}{\partial x_{j}
      \partial y_j}-\frac{1}{2}(\frac{\partial L}{\partial x_1}
      + \frac{\partial L}{\partial y_0})
      &\mbox{for} & j=k\\[2mm]
      \frac{\partial^2L}{\partial x_{j}\partial y_{k}}
      &\mbox{for} & j\ne k\end{array}\right.
\end{eqnarray*}

\begin{eqnarray*}
   (M_{\tau\tau})_{jk} = (M_{\tau\tau})_{kj}
   &=&\frac{\partial^2(L\circ \psi )}{\partial \tau _{j}
    \partial\tau_{k}}= \left\langle L''\frac{\partial \psi}
    {\partial \tau_{j}}(0),\frac{\partial \psi}{\partial \tau _{k}}(0)
    \right\rangle
  +\left\langle L', \frac{\partial^2\psi}
    {\partial\tau_{j}\partial\tau _{k}}(0)\right\rangle\\
  &=& \left\{\begin{array}{lcl}\langle L''(0,e_j),(0,e_j)\rangle -
         \langle L', (0,e_1)\rangle &\mbox{for} & j=k\\
        \langle L''(0,e_j),(0,e_k)\rangle
         &\mbox{for} & j\ne k\end{array}\right. \\
  &=& \left\{\begin{array}{lcl}\frac{\partial^2L}{\partial
           y_{j}^2}-\frac{\partial L}{\partial y_1}
          &\mbox{for} & j=k\\[2mm]
         \frac{\partial^2L}{\partial y_{j}\partial y_{k}}
         &\mbox{for} & j\ne k\end{array}\right.  ,
\end{eqnarray*}
for $2\le j\le n$,
\begin{eqnarray*}
  (M_{\sigma\theta})_{j1} = (M_{\theta\sigma})_{1j}
  &= &\frac{\partial^2(L\circ \psi )}{\partial\sigma _{j}
        \partial \theta}
    = \left\langle L''\frac{\partial \psi}{\partial \sigma_j}(0),
        \frac{\partial \psi}{\partial \theta}(0)\right\rangle +
       \left\langle L', \frac{\partial^2\psi}{\partial\sigma_{j}
       \partial\theta}(0)\right\rangle\\
  &=& \frac{1}{\sqrt 2}\langle L''(e_j,0),(e_1,-e_0)\rangle \ = \
         \frac{1}{\sqrt 2}\left(\frac{\partial^2L}{\partial x_{j}
         \partial x_1}-\frac{\partial^2L}{\partial x_{j}\partial y_0} \right),
\end{eqnarray*}

\begin{eqnarray*}
   (M_{\tau\theta})_{j1} = (M_{\theta\tau})_{1j}
   &= &\frac{\partial^2(L\circ \psi )}{\partial\tau _{j}\partial \theta}
     = \left\langle L''\frac{\partial \psi}{\partial \tau_j}(0),
        \frac{\partial \psi}{\partial \theta}(0)\right\rangle
     +\left\langle L', \frac{\partial^2\psi}{\partial\tau_{j}
        \partial\theta}(0)\right\rangle\\
  &=&\frac{1}{\sqrt 2}\langle L''(0,e_j),(e_1,-e_0)\rangle \
    = \ \frac{1}{\sqrt 2}\left(\frac{\partial^2L}
        {\partial y_{j}\partial x_1}-\frac{\partial^2L}
       {\partial y_{j}\partial y_0} \right),
\end{eqnarray*}
and finally
\begin{align}
     M_{\theta \theta}
    \;=\;&  \frac{\partial^2(L\circ \psi )}{\partial \theta^2}
    \;=\;\left\langle L''\frac{\partial \psi}{\partial \theta}(0),
     \frac{\partial \psi}{\partial \theta}(0)\right\rangle
    +\left\langle L', \frac{\partial^2\psi}{\partial\theta^2}(0)
       \right\rangle \notag\\
   =\;&\frac{1}{2}\left(\langle L''(e_1,-e_0),(e_1,-e_0)\rangle -
       \langle L', (e_0,e_1)\rangle \right)\notag\\
   =\;&\frac{1}{2}\left(\frac{\partial^2L}{\partial x_1^2}
            -2\frac{\partial^2 L}{\partial x_1\partial y_0}
           +\frac{\partial^2L}{\partial y_0^2}-\frac{\partial L}
           {\partial x_0}-\frac{\partial L}{\partial y_1}  \right).\tag*{\qed}
\end{align}
\end{corollary}

\bigskip

\proofof{Lemma~\ref{condincond}} We have:

\begin{equation}\label{espcondg}
\E \big( g(X)\, /\,XY+Z=u, Y=0\big)=\int_{\R^{p\times q}}g(x)\frac{p_{X,Y,XY+Z}(x,0,u)}{p_{Y,XY+Z}(0,u)}~dx
\end{equation}
since
$$
\frac{p_{X,Y,XY+Z}(x,0,u)}{p_{Y,XY+Z}(0,u)}
$$
is the conditional density of $X$ at the point $x$, given that $Y=0, XY+Z=u$.\\
Now, the density $p_{X,Y,XY+Z}(x,y,u)$ is easily computed from the change of variables formula (using the independence of $X,~Y,~Z$), obtaining:
$$
p_{X,Y,XY+Z}(x,y,u)=p_X(x)p_Y(y)p_Z(u-xy).
$$
This also implies
$$
p_{Y,XY+Z}(0,u)=\int_{\R^{p\times q}}p_{X,Y,XY+Z}(x,0,u)~dx
=p_Y(0)p_Z(u).
$$
Replacing $p_{Y,XY+Z}(0,u)$ by $p_Y(0)p_Z(u)$
in (\ref{espcondg}), we get:
$$
\E \big( g(X)\, /\, XY+Z=u, Y=0\big)=\int_{\R^{p\times q}}g(x)p_X(x)~dx=\E \big( g(X)\big)
$$
\eproof

\medskip

\proofof{Lemma~\ref{polcaraccasi}}  Write the Taylor expansion of $\det (C_q(\lambda))$ at
$\lambda =0$ and compute the successive derivatives at this point.
\eproof

\medskip

\proofof{Lemma~\ref{detbasico}}  We note that $ Q=M\,M^t $ where
$$
   M:=\stackrel{\bajo n\hspace*{10pt} n}
        {\left(\begin{array}{c|c}
       A&B\\ [3pt] \hline \alto C&0
       \end{array}\right)}
       \begin{array}{l}
       {\scriptstyle k} \\ [4pt]
       {\scriptstyle n-1}
       \end{array}.
$$
Applying the Cauchy-Binet formula~\eqref{CB}, we get

\begin{equation}\label{CBforQ}
\det(Q)=\sum_{\# (S')=k+n-1}  \big(\det (M^{S'} )\big)^2,
\end{equation}
where the sum is over all choices of $k+n-1$ columns of $M$.
\\

\noindent  We fix such an $S'$. It is easy to see, performing  a
Laplace expansion with respect to the first $k$ rows of the obtained
matrix, that if we take strictly more than $k$ columns in the
$n$-columns right block corresponding to $B$, then $\det (M^{S'}
)=0$. This is because in this expansion there will always remain  a
zero column. Therefore, we can only choose up to $k$ columns in the
right block, i.e. there are two cases: we choose all the $n$ columns
in the left block and $k-1$ columns in the right block, or we choose
$n-1$ columns in the left block and $k$ columns in the right block.

\smallskip
\noindent \emph{Case 1:} $ M^{S'}$ is of  the form:
$$
M^{S'}= \stackrel{\bajo n\hspace*{10pt} k-1}
        {\left(\begin{array}{c|c}
        A&B^{S}\\ [3pt] \hline
       \alto C&0
       \end{array}\right)}
       \begin{array}{l}
       {\scriptstyle k} \\ [4pt]
       {\scriptstyle n-1}
       \end{array}\ \in \R^{(k+n-1)\times (k+n-1)}.
$$
 Here $S$ is the set of $(k-1)$ columns of $B$ that we kept.
 Again using  Laplace expansion with respect to the  last $k-1$
columns of $M^{S'}$, we see that each non-zero determinant
corresponds to suppressing a row --say row $i$-- of $B^S$, times the
determinant of its complementary matrix which is equal to the $i$-th
row of $A$ added to $C$. Finally, expanding  this last matrix by the
$i$-th row of $A$, we obtain:
$$\det(M^{S'})=(-1)^{n(k-1)}\sum_{i=1}^k(-1)^{k-i}\det(B_{\overline i}^S \sum_{j=1}^n (-1)^{j-1}a_{ij}\det(C^{\overline j}),$$
where $\overline{i}$ and $\overline{j}$ denote the complementary
rows or columns, accordingly.

\smallskip
\noindent  \emph{Case 2:} $ M^{S'}$ is of  the following form for
some $j$ which corresponds to the suppressed
 column of $A$
 and $S$ is a choice of $k$ columns of $B$:
$$
M^{S'}= \stackrel{\bajo n-1\hspace*{10pt} k}
        {\left(\begin{array}{c|c}
        A^{\overline{j}}&B^{S}\\ [3pt] \hline
        \alto C^{\overline j}&0
       \end{array}\right)}
       \begin{array}{l}
       {\scriptstyle k} \\ [4pt]
       {\scriptstyle n-1}
       \end{array}\ \in \R^{(k+n-1)\times (k+n-1)}.
$$
Then, permuting the two blocks of rows and since the obtained matrix
is block-diagonal, we get
$\det(M^S)=(-1)^{k(n-1)}\det(C^{\overline{j}})\det(B^{S})$.\\

\noindent  Therefore,  the sum in (\ref{CBforQ}) for all $S'$ in
Case 2 gives:
$$
\sum _{j=1}^n \big( \det (C^{\overline{j}}) \big)^2\sum _{\#
(S)=k}\big( \det (B^{S} ) \big)^2= \det (CC^t)\det (BB^t),
$$
again by the Cauchy-Binet formula \eqref{CB}.
The statement  follows from adding up over all $S'$ in Cases~1
and~2. \eproof

\medskip

\proofof{Lemma~\ref{momlambda}}  The proof is based on the following bound for the tails of
the probability distribution of $\overline{\lambda}$. For $t>0$ one
has (see for example~\cite{DaSz} and references therein):

\begin{equation*}
\P (\overline{\lambda} \geq 2+\sqrt{2}~t)< \exp \big(-\frac{nt^2}{2}
\big).
\end{equation*}
Therefore, since  $\ell \leq n$,

\begin{equation*}
\aligned \E \big( \overline{\lambda}^\ell \big)\ &=\
\int_0^{+\infty}\P (\overline{\lambda}^\ell >x)\,dx \ =
\ \int_0^{+\infty}\P (\overline{\lambda} >y)~\ell y^{\ell -1}\,dy\\
\ & \leq \ \int_0^4\ell y^{\ell -1}\,dy\ +\ \int_4^{+\infty}\ell\,
y^{\ell -1}\,\exp \Big( -\frac{n}{2} \cdot \frac{(y-2)^2}{2}\Big)\,
dy\\
\ &\leq \ 4^\ell\ +\ \int_{\sqrt{2n}}^{+\infty}\ell
\,\sqrt{\frac{2}{n}}\,\Big(\sqrt{\frac{2}{n}}u+2 \Big)^{\ell
-1}\,\exp \Big(-\frac{u^2}{2} \Big)\,
du\\
\ &=\ 4^\ell\ +\ \ell \,\sqrt{\frac{2}{n}}\,2^{\ell
-1}\,\int_{\sqrt{2n}}^{+\infty}\Big(1+\frac{u}{\sqrt{2n}}
\Big)^{\ell -1}\,
\exp \Big(-\frac{u^2}{2} \Big)\,du\\
&\leq 4^\ell+\ell \,\sqrt{\frac{2}{n}}\,2^{\ell
-1}\,\int_{\sqrt{2n}}^{+\infty}\exp
\big( u\sqrt{\frac{n}{2}}-\frac{u^2}{2}\big)\,du \quad \mbox{since} \quad 1+x\le \exp(x)\\
\ &= \ 4^\ell \ + \ \ell \,\sqrt{\frac{2}{n}}\,2^{\ell
-1}\,\exp(n/4)\,\int_{\sqrt{n/2}}^{+\infty}\exp\big(
-\frac{y^2}{2}\big)\,
dy\\
\ &\leq \  4^\ell\ +\ \ell \,\sqrt{\frac{2}{n}}\,2^{\ell
-1}\,\exp\Big(\frac{n}{4}\Big)\,\sqrt{\frac{2}{n}}\,\exp
\Big(-\frac{n}{4}\Big)\ \leq \ 2\cdot 4^\ell,
\endaligned
\end{equation*}
where in the last line we used that
\begin{equation}\tag*{\qed}
\int_{a}^{+\infty}
\exp\Big(-\frac{y^2}{2}\Big)\,dy < \int_{a}^{+\infty}
\frac{y}{a}\,\exp\Big(-\frac{y^2}{2}\Big)\,dy
=\frac{1}{a}\exp\Big(-\frac{a^2}{2}\Big).
\end{equation}

\bigskip
\noindent{\bf Acknowledgment.}\quad We are thankful to the anonymous referee for his many suggestions that helped us improving the
presentation of this text.
\bigskip

{\small

}
\end{document}